\documentclass{iopart}

\usepackage{epsfig}
\expandafter\let\csname equation*\endcsname\relax
\expandafter\let\csname endequation*\endcsname\relax

\usepackage{amsmath}
\usepackage{amsopn}
\usepackage{amssymb}
\usepackage{cite}
\usepackage{amsthm}
\usepackage{graphicx}
\usepackage{epstopdf}
\usepackage{url}
\usepackage{algorithmic} 
\usepackage{algorithm}
\usepackage{verbatim}
\usepackage{multirow}
\usepackage{wrapfig}
\usepackage{array}
\usepackage{multirow}
\usepackage{mathrsfs}

\usepackage{graphicx}
\usepackage{bm}

\theoremstyle{definition}
\newtheorem{theorem}{Theorem}

\newtheorem{corollary}[theorem]{Corollary}
\newtheorem{lemma}[theorem]{Lemma}

\newtheorem{proposition}[theorem]{Proposition}

\DeclareMathOperator*{\argmin}{argmin \; }
\DeclareMathOperator*{\minimize}{minimize \quad}
\DeclareMathOperator*{\subto}{subject\, to \quad}

\newcommand{\reals}{\ensuremath{\mathbb{R}}}

\newcommand{\norm}[1]{ \left \Vert #1 \right \Vert }

\newcommand{\Bfun}{\mathscr{B}}

\newcommand{\Xset}{\mathcal{X}}
\newcommand{\Yset}{\mathcal{Y}}
\newcommand{\Zset}{\mathcal{Z}}

\newcommand{\Active}{\mathcal{A}}
\newcommand{\Inactive}{\mathcal{I}}

\newcommand{\Trans}{T}

\newcommand{\bSeries}{\boldsymbol{\beta}}

\newcommand{\sepr}{\,|\,}

\newcommand{\Proj}{\mathcal{P}}
\newcommand{\boxSet}{\mathcal{B}}

\newcommand{\beq}{\begin{equation}}
\newcommand{\eeq}{\end{equation}}
\newcommand{\bbm}{\begin{bmatrix}}
\newcommand{\ebm}{\end{bmatrix}}

\newcommand{\str}{\leftarrow}
\newcommand{\kron}{\raisebox{1pt}{\ensuremath{\:\otimes\:}}} 
\newcommand{\BFL}{\begin{flalign}}
\newcommand{\EFL}{\ensuremath{\end{flalign}}}
\newcommand{\pinv}{+}

\begin{document}
\title[Generalizing Variable Elimination Beyond Least Squares]{A Generalization of Variable Elimination for Separable Inverse Problems Beyond Least Squares}
\author{Paul Shearer, Anna C. Gilbert}
\date{\today}

\begin{abstract}
In linear inverse problems, we have data derived from a noisy linear transformation of some unknown parameters, and we wish to estimate these unknowns from the data. Separable inverse problems are a powerful generalization in which the transformation itself depends on additional unknown parameters and we wish to determine both sets of parameters simultaneously. When separable problems are solved by optimization, convergence can often be accelerated by elimination of the linear variables, a strategy which appears most prominently in the variable projection methods due to Golub and Pereyra. Existing variable elimination methods require an explicit formula for the optimal value of the linear variables, so they cannot be used in problems with Poisson likelihoods, bound constraints, or other important departures from least squares. 

To address this limitation, we propose a generalization of variable elimination in which standard optimization methods are modified to behave as though a variable has been eliminated. We verify that this approach is a proper generalization by using it to re-derive several existing variable elimination techniques. We then extend the approach to bound-constrained and Poissonian problems, showing in the process that many of the best features of variable elimination methods can be duplicated in our framework. Tests on difficult exponential sum fitting and blind deconvolution problems indicate that the proposed approach can have significant speed and robustness advantages over standard methods. 

\end{abstract}
\section{Introduction}
In linear inverse problems we are given a vector of noisy data $b \in \reals^m$ generated by the linear model 
$b = Az + \epsilon$, where $A \in \reals^{m \times c}$ is a known matrix, $\epsilon$ is a zero mean noise vector, and $z \in \reals^{N_z}$ is an unknown vector with $N_z = c$ entries we wish to estimate. In separable inverse problems, $A$ is not known exactly, but depends on another set of parameters $y \in \reals^{N_y}$:
\beq
	 b = A(y)z + \epsilon.
	\label{cond_lin_eqn}
\eeq
The problem is now to determine the full set of $N = N_y + N_z$ parameters $x \triangleq (y,z)$.

Many scientific inverse problems are separable. In time-resolved spectroscopy and physical chemistry, data are often modeled as a weighted sum of several (possibly complex) exponentials with unknown decay rates \cite{Yamaoka:sumOfExpAIC:1978, Mullen:varProInSpect:2009}. Determining the weights and decay rates simultaneously is a separable inverse problem. Other examples include image deblurring with an incompletely known blur kernel \cite{Campisi:BlindImageDeconvolution:2007} and tomographic reconstruction from incomplete geometric information \cite{Chung:2010}. Many more examples can be found in \cite{Golub:separablenonlinear:2002,hansen2013least,pereyra2010exponential}. 

Separable problems frequently have additional exploitable structure. In this paper, we will be particularly interested in problems with multiple measurement vectors generated by applying a single linear transformation to $n$ different vectors of linear coefficients. In this case, the data and coefficient vectors can be represented by matrices $B \in \reals^{m \times n}$ and $Z \in \reals^{c \times n}$, and we have
\beq
	B = A(y)Z + E.
	\label{cond_lin_mat_eqn}
\eeq
This problem, also known as a \emph{multiple right-hand sides} or \emph{multi-way data} problem \cite{golub1979extensions,kaufman1992separable,mullen2007timp}, occurs when a system is repeatedly observed under varying experimental conditions \cite{Mullen:varProInSpect:2009}. 

An inverse problem is generally solved by seeking parameter values that balance goodness of fit with conformity to prior expectations. In this paper we focus on constrained maximum likelihood problems, where we choose a goodness of fit function $L(A(y)z)$ measuring discrepancy between $A(y)z$ and $b$ and a set $\Xset = \Yset \times \Zset$ representing known constraints on $y$ and $z$, such as nonnegativity. We seek the parameter values that minimize the discrepancy subject to the constraints by solving
\beq
	\min_{y \in \Yset \,z \in \Zset} \bigg\{ F(y,z) \triangleq L(A(y)z)  \bigg\}.
	\label{fullObj}
\eeq
Penalty functions such as $\ell_p$ norms on $y$ and $z$ may also be incorporated into $F(y,z)$, and while our techniques are relevant to this case, it is not specifically addressed here. For the goodness of fit function we use the negative log-likelihood $L(\mu) = -\log p(b \sepr \mu)$, where the likelihood function $p(b \sepr \mu)$ is the probability that $b = \mu + \epsilon$ and is determined by the distribution of $\epsilon$. Least squares problems result from assuming standard Gaussian distributed noise, so that $L(\mu) = \frac{1}{2} \norm{\mu - b}^2$, but Poissonian and other likelihoods frequently arise.

Unconstrained least squares problems are generally easiest to solve, and many powerful optimization ideas were first developed for this case \cite{NocedalWright:NumericalOptimization:2006}. However, unconstrained least squares solutions are not always satisfactory, and much better solutions can often be found using nonnegativity constraints, Poisson likelihoods, or other departures from ordinary least squares. Many physical quantities must be nonnegative, and enforcing this constraint can reduce reconstruction error \cite{bardsley2004nonnegatively} and help the optimizer avoid unphysical answers \cite{SimaVanHuffel:varProBndMRS:2007}. A Poisson process is often the best model for a stream of particles entering a detector, and in the low-count limit the Poisson and Gaussian distributions are very different. In this case Poissonian optimization usually gives significantly better parameter estimates than least squares, a fact of fundamental importance in astronomy \cite{bardsley2004nonnegatively,vogelComputationalMethods, bertero2009image}, analytical chemistry \cite{maus2001experimental}, and biochemistry \cite{Larson:2010fk,Laurence:2010:mlePoisson}, where information must be extracted efficiently from a trickle of incoming photons.  This paper is concerned with advancing the state of the art for problems beyond least squares.

\subsection{Existing optimization methods}
We will focus on optimization methods employing \emph{Newton-type} iterations. While other powerful methods exist for inverse problems, Newton-type methods enjoy very general applicability, attractive convergence properties, scalability under favorable conditions, and robustness against ill-conditioning and nonconvexity \cite{NocedalWright:NumericalOptimization:2006}. Given a smooth function $f(u)$, a constraint set $\mathcal{U} \subset \reals^{N_u}$, and an initial point $u^0 \in \mathcal{U}$, a Newton-type method generates a sequence of iterates $u^1,u^2,\ldots$ which hopefully converge to the minimizer of $f(u)$ in $\mathcal{U}$, or at least a stationary point. Line search methods, which will be the focus of this paper, generally use the following update procedure to go from $u^k$ to $u^{k+1}$ \cite{NocedalWright:NumericalOptimization:2006, kelley1987iterative}:

\begin{enumerate}
\item \emph{Search direction:} A search direction $\Delta u$ is calculated by solving a Newton-type system of the form $B \Delta u = -g$, 
	where $g$ is determined from the gradient $\nabla f(u^k)$ and $B$ is a \emph{Hessian model} approximating $\nabla^2 f(u^k)$. 
	Both $g$ and $B$ may be modified by information from constraints and previous iterates.	
\item \emph{Trial point calculation:} The step $\Delta u$ determines a search path $u_p(s)$, parametrized by a step size $s > 0$, 
	from which a trial point $\bar{u}$ is selected. This is generally a straight-line path modified to maintain feasibility with respect
	to constraints or hedge against a bad search direction.	
\item \emph{Evaluation and decision:} If moving to the trial point produces a sufficient decrease in the objective,
	 we set $u^{k+1} = \bar{u}$. Otherwise, another trial point is constructed, possibly along a new direction $\Delta u$,
	 and the process is repeated.
\end{enumerate}

This update procedure is used in the service of some larger strategy for optimizing $F(y,z)$. To understand the strategies typically used, it is helpful to first consider strategies for solving the block-structured system $B \Delta x = - g$. This system has the block expansion 
\beq
	\begin{bmatrix}
		B_{yy} & B_{yz} \\
		B_{zy} & B_{zz}
	\end{bmatrix}
	\begin{bmatrix}
		\Delta y \\
		\Delta z
	\end{bmatrix}
	=
	-
	\begin{bmatrix}
		g_y \\
		g_z
	\end{bmatrix},
	\label{fullSystem}
\eeq	
and is typically solved in one of three ways. (In the following, the product $M^{-1} w$ should be interpreted as a directive to solve $M v = w$ for $v$ rather than to compute $M^{-1}$ explicitly, and when we speak of inversion we refer to this directive.) 
\begin{enumerate}
\item \emph{Full matrix, all-at-once.} We solve the whole system at once by QR or Cholesky factorization in medium-scale problems, and by conjugate gradients (CG) in very large-scale problems. 
\item \emph{Block Gauss-Seidel.} We converge to a solution by iterative updates of the form 
\begin{align}
	 \Delta y^{j+1} &= -B_{yy}^{-1}(g_y - B_{yz} \Delta z^j) \\
	 \Delta z^{j+1} &= -B_{zz}^{-1}(g_z - B_{zy} \Delta y^j).
\end{align}
Gauss-Seidel is fast provided that $B_{yy}$ and $B_{zz}$ are much easier to invert than all of $B$ and a block diagonal approximation of $B$ is reasonably accurate, but may be arbitrarily slow to converge otherwise \cite{saad1996iterative}. 
\item \emph{Block Gaussian elimination.} By solving for $\Delta z$ in the bottom row of \eqref{fullSystem} 
and substituting the result into the top row equation, we decompose \eqref{fullSystem} as
\begin{subequations}
\begin{align}
	B_s \Delta y
	&= -g_y + B_{yz} B_{zz}^{-1} g_z \label{y_schur_eqn} 
	\\
	B_{zz} \Delta z
	&= -g_z - B_{zy} \Delta y, 
	\label{z_schur_eqn}
\end{align}
\label{schur_system}%
\end{subequations}
where $B_s \triangleq B_{yy} - B_{yz} B_{zz}^{-1} B_{zy}$ is the Schur complement of $B_{zz}$ in $B$ \cite{boyd2004convex}. We construct the matrix $B_s$ explicitly, solve for $\Delta y$ in \eqref{y_schur_eqn}, then plug the result into \eqref{z_schur_eqn} to solve for $\Delta z$.
\end{enumerate}

\noindent Assuming $B$ is positive definite, all three of these linear solvers can be interpreted as a method for minimizing the quadratic form $\frac{1}{2} \Delta x^\Trans B \Delta x + g^\Trans \Delta x$. Each of them can also be generalized to an update strategy for the nonquadratic problem \eqref{fullObj}, as follows:

\begin{enumerate}
\item \emph{Full update:} We update $y$ and $z$ simultaneously using a step derived from solving the full system \eqref{fullSystem}. 
Any classical Newton-type method applied directly to $F(y,z)$ falls into this category \cite{NocedalWright:NumericalOptimization:2006}. 
\item \emph{Alternating update:} We make one or more updates to $z$ with $y$ fixed, then to $y$ with $z$ fixed, alternating until convergence \cite{bertsekas1999nonlinear}. Alternating methods now have well-developed convergence theory even with inexact alternating minimizations, and their iterations do not necessarily require matrix factorizations \cite{bonettini2011inexact,grippo1999globally}. As such, they may be the only tractable choice for certain large-scale and highly non-parametric problems such as nonnegative matrix factorization. However, like Gauss-Seidel, alternating methods can converge slowly \cite{Buchanan05,ruhe1980algorithms} and are generally preferable only when full updates are computationally expensive or intractable. In this paper we will focus on problems where full update methods are tractable, so alternation will not be considered further. 
\item \emph{Reduced update:}  We determine the optimal $z$ value given $y$,
\beq
	z_m(y) = \argmin_{z \in \Zset} L(A(y)z),
	\label{zOpt}
\eeq
and substitute it into \eqref{fullObj}, giving an equivalent reduced problem
\beq
	\min_{y \in \Yset}  \bigg\{ F_r(y) \triangleq F(y,z_m(y)) \bigg\}
	\label{reducedObj}
\eeq
to which the Newton-type iteration is applied, meaning that we set $f(u) = F_r(y)$ instead of $F(y,z)$. The resulting update has a nested structure: an outer optimizer computes the search direction $\Delta y$ and trial point $\bar{y}$, while an inner optimizer calculates $z$ by solving \eqref{zOpt} whenever the outer one asks for the value of $F_r(y)$ or its derivatives. 
\end{enumerate}

Most reduced update methods are variations on the variable projection algorithm of Golub and Pereyra \cite{golub1973differentiation,Golub:separablenonlinear:2002}, which applies to the case of unconstrained separable least squares. In this case we have $F(y,z) = \frac{1}{2} \norm{A(y)z - b}^2$ and $z_m(y) = A(y)^\pinv b$, where $X^\pinv$ denotes the Moore-Penrose pseudoinverse. Substituting $z_m(y)$ into $F(y,z)$ yields $F_r(y) = \frac{1}{2} \norm{-P_A^\perp b}^2$, where $P_X^\perp = I - XX^\pinv$ denotes the projection onto $\text{range}(X)^\perp$, and the $y$ in $A(y)$ has been suppressed. Golub and Pereyra proposed using a Gauss-Newton method to optimize $F_r(y)$; this requires the Jacobian for the reduced residual $- P_A^\perp b$, which they derived by differentiation of pseudoinverses. This idea can also be extended to accommodate linear constraints on $z$.

The efficiency of variable projection in highly ill-conditioned curve fitting and statistical inference problems is theoretically and empirically well-attested \cite{Golub:separablenonlinear:2002,ruhe1980algorithms,smyth1996partitioned,o2012variable}. Variable projection is also useful for problems with multiple measurement vectors  \cite{golub1979extensions,kaufman1992separable,mullen2007timp}, as in this case $A(y)$ is block diagonal, so necessary pseudoinverses and derivatives may be efficiently computed blockwise. Other methods based on variable elimination can speed up the solution of large-scale image and volume reconstruction problems if the pseudoinverse and derivatives can be computed quickly \cite{Vogel:PhaseDiv:1998, gilles2002computational,fan2011efficient,Chung:2010}.

Given the efficiency of variable elimination methods in separable least squares problems, one might hope to derive an extension with similar advantages to problems beyond least squares. However, such an extension runs into several difficulties. First, in problems beyond least squares there is generally no analytical formula for $z_m(y)$, and computing it is often computationally expensive. Second, if inequality constraints or nonsmooth penalties are imposed on $z$, then $z_m(y)$ will be a nonsmooth function with unpredictable properties, so that the reduced problem may be even more difficult than the original. Third, without a formula for $z_m(y)$ it is unclear how to compute $Dz_m(y)$, which is needed for a fast-converging second-order method.

\subsection{Our contribution} \label{our_contribution}

Variable elimination does not seem to generalize easily to non-quadratic and constrained problems, but there are many efficient and robust full update methods for such problems \cite{NocedalWright:NumericalOptimization:2006}. This fact suggests that we might arrive at a generalization more easily from the other direction, by making existing full update methods resemble reduced update methods more closely. In this paper we explore the resulting \emph{semi-reduced} update methods, explain how they relate to full and reduced update methods, describe when they are useful, and validate our claims with numerical experiments on hard inverse problems similar to those encountered in practice.

In \S\ref{global_conv} we show how to transform a full update method into a reduced method without an explicit formula for $z_m(y)$. We begin by applying two specific changes to a full update method: first, use block Gaussian elimination instead of an all-at-once solver, and second, adjust every new trial point's $z$ coordinate to a better value before the trial point is evaluated. This second technique, which we call \emph{block trial point adjustment}, is depicted graphically in Fig.~\ref{visual_mod}, \emph{right}. We call a full update method thus modified a semi-reduced method. Reduced methods are obtained from semi-reduced methods by requiring that the adjustment be optimal, which enables us to simplify the method by omitting computations of $\nabla_z F$ and the search direction $\Delta z$. We show reduced Newton and variable projection methods can be derived in this way.  In \S\ref{bound_constrained_opt}, we propose and prove convergence of a semi-reduced method that allows for nonquadratic likelihoods and bound constraints on $z$, which has been posed as an open problem by multiple authors \cite{Chung:2010, Mullen:varProInSpect:2009}. 
 
 The description of reduced and semi-reduced methods as modifications of full update methods allows us to predict when the former have advantages over the latter. Block Gaussian elimination is most effective when $B_{zz}$ is easier to invert than all of $B$, for example when $B_{zz}$ has block diagonal (Fig.~\ref{visual_mod}), Toeplitz, banded, or other efficiently invertible structure. Block trial point adjustment should yield an efficiency gain when the computational burden of the adjustment subproblems is outweighed by an increase in convergence rate. This may occur when the graph of the objective contains a narrow, curved valley like that shown in Fig.~\ref{visual_mod}. 

To test these predictions we select problems where we expect semi-reduced methods to have an advantage,  design methods for these problems using the semi-reduced framework, then compare the semi-reduced methods to standard full update methods. In \S \ref{sec:block_lin_algs} we derive linear algebra techniques that use block Gaussian elimination to exploit block structure or spectral properties of $B$, and in \S \ref{exp_sum_fit} and \S \ref{semiblind} we study two problems of scientific interest where these techniques have advantages over standard full-matrix methods. In \S \ref{model_adjustment_prob} we consider a toy blind deconvolution problem where block trial point adjustment leads to a significant increase in convergence rate due to a curved valley geometry. We conclude that semi-reduced methods can have significant advantages over full update methods under the predicted conditions.

\subsection{Related work} \label{related_work}

While the relationship between full and reduced update methods has been explored several times, the relationship established here is a major extension of previous work. In \cite{ruhe1980algorithms} Ruhe and Wedin developed the connection between full and reduced update Newton and Gauss-Newton methods, and semi-reduced methods are described by Smyth as \emph{partial Gauss-Seidel} or \emph{nested} methods in \cite{smyth1996partitioned}. Our work extends theirs in that we consider general Newton-type methods, nonquadratic likelihoods, and the effect of globalization strategies, such as line search or trust regions, which ensure convergence to a stationary point from arbitrary initialization. A very general theoretical analysis of the relationship between the full and reduced problems is given in \cite{parks1985reducible}, but there is little discussion of practical algorithms and no mention of semi-reduced methods.

Structured linear algebra techniques such as block Gaussian elimination are known to be useful \cite{chan1985approximate,zhang2005schur}, but they are underutilized in practice. This is apparent from the fact that most optimization codes employ a limited set of broadly applicable linear algebra techniques \cite{NocedalWright:NumericalOptimization:2006}, and very few are designed to accommodate user-defined linear solvers such as the ones we propose in \S \ref{sec:block_lin_algs}. We contend that since significant speed gains are attainable with special linear solvers,  optimization algorithm implementations should accommodate user-customized linear algebra by adding appropriate callback and reverse communication protocols.

Trial point adjustment is a key idea in the two-step line search and trust region algorithms of \cite{Conn:1998:techreport} and \cite{conn:twoStep:1999}. General convergence results are proven in \cite{absil2009accelerated} for `accelerated' line search and trust region methods employing trial point adjustment. These works are not concerned with separable inverse problems or the relationship with reduced methods.

Extensions of variable projection beyond unconstrained least squares have been proposed, in particular to accommodate bound constraints on $z$ \cite{SimaVanHuffel:varProBndMRS:2007, Nagy:nonnegVarPro:2012}. Their approach is to apply a Newton-type method to minimize $\tilde{F}_r(y) = F(y,\tilde{z}_m(y))$, an approximation of $F_r(y) = F(y,z_m(y))$ obtained by computing $z_m(y)$ approximately using a projected gradient or active set method. This approach can work well, but it has several theoretical and practical downsides. First, it has not been extended to nonquadratic likelihoods; second, computing $z_m(y)$ can be very expensive, and the precision required is unclear; third, an appropriate Hessian model is not obvious and must be obtained by ad-hoc heuristics or finite differences; and fourth, there has been no attempt at global convergence results. In contrast, our approach works on nonquadratic likelihoods; it provides the option of approximating $z_m(y)$ to any desired precision without danger of sacrificing convergence; one may use the same standard Hessian models used in full update methods, with exact derivatives if desired; and we prove a global convergence result for our method.
\begin{figure}
	\begin{center}
	\includegraphics[scale=0.8]{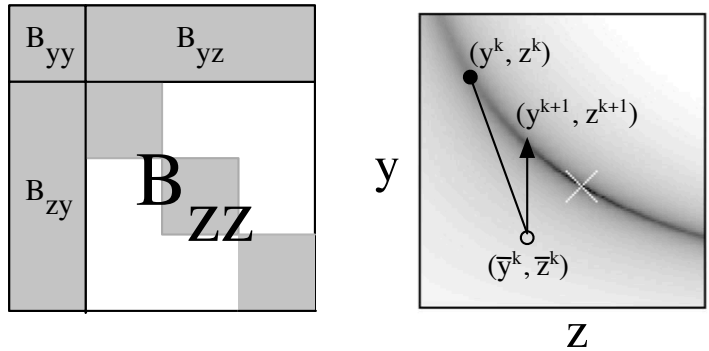}
	\end{center}
	\caption{Situations where block gaussian elimination and trial point adjustment may be useful.
	\emph{Left:} A `block arrow' matrix $B$ containing a block diagonal submatrix $B_{zz}$ is well-suited for inversion by block Gaussian
	elimination. This type of matrix arises in separable problems with multiple measurement vectors. 
	\emph{Right:} Graph of an objective $F(y,z)$ exhibiting a narrow, curved valley; the minimum is marked with an X. 
	Superimposed are a sample iterate $(y^k,z^k)$ and an initial trial point $(\bar{y}^k,\bar{z}^k)$ 
	that fails a sufficient decrease test. By adjusting this point's $z$ coordinate to the minimum of $F(\bar{y}^k,z)$, we obtain
	a new trial point $(y^{k+1},z^{k+1})$ that provides sufficient decrease to be accepted as an update.}
	\label{visual_mod}
\end{figure}

\section{Semi-reduced methods as a generalization of variable elimination} \label{global_conv}
In this section we show that a full update algorithm may be transformed into a reduced update (variable elimination) algorithm by introducing block Gaussian elimination and an optimal block trial point adjustment, then simplifying the resulting algorithm to remove unnecessary computation. Semi-reduced methods are those obtained halfway through this process, after the block techniques are imposed but before the simplification. We will describe the transformation process for unconstrained Newton-type line search algorithms, but it can be done with other types of algorithms too.  

\subsection{Semi-reduced methods and their simplification under optimal adjustment}
 We begin the move towards semi-reduced methods by defining a standard unconstrained line search algorithm, then adding trial point adjustment. Let $f(u)$ be a twice-differentiable function and $\Bfun(u) \in \reals^{N_u \times N_u}$ the \emph{Hessian model}, a positive definite matrix-valued function approximating $\nabla^2 f(u)$.
 
 Given an iterate $u^k$, we obtain the update $u^{k+1}$ by the following procedure. We begin by setting $g = \nabla f(u^k)$, $B = \Bfun(u^k)$, and determining the search direction $\Delta u$ by solving $B \Delta u = -g$. The search direction determines a line $u_p(s) = u^k + s \Delta u$ of potential trial points parametrized by step size $s$, and we set $u^{k+1}$ by choosing one that satisfies the \emph{sufficient decrease} condition
\beq
	f(u_p(s)) - f(u^k) \leq \delta g^\Trans (u_p(s) - u^k) = \delta g^\Trans( s \Delta u),
	\label{armijo}
\eeq
for a fixed $\delta \in (0,1/2)$. One can generally ensure convergence by picking a step size that obeys this condition and is not too small. Such a step size can be obtained by backtracking: we set $s = \alpha^j$ and try $j = 0,1,2,\ldots$ until \eqref{armijo} is satisfied. 

To incorporate trial point adjustment into this update procedure, we assume we are given an \emph{adjustment operator} $u_d(u)$ such that $f(u_d(u)) \leq f(u)$ for any input $u$. We then replace $u_p(s)$ with $u_d(u_p(s))$ on the left hand side of \eqref{armijo}, obtaining Alg.~\ref{mod_newt_method_ls}. (Note that the standard full update method may be recovered by setting $u_d(u) = u$.) Global convergence of Alg.~\ref{mod_newt_method_ls} to a stationary point is guaranteed by the following theorem: 
\begin{theorem} \label{global_conv_arm}
Assume that $f(u)$ is bounded below, $\nabla f(u)$ is Lipschitz continuous with bounded Lipschitz constant, and the matrices $\Bfun(u)$ are symmetric positive definite with eigenvalues bounded away from zero and infinity. Then
\beq
	\lim_{n \rightarrow \infty} \nabla f(u^k) = 0,
\eeq
and any limit point of $(u^k)_{k = 0}^\infty$ is a stationary point.
\end{theorem}
\noindent This theorem is proven in \cite{kelley1987iterative} for the standard algorithm without trial point adjustment, while the extension for algorithms including trial point adjustment is given in \cite{absil2009accelerated}. Informally, trial point adjustment does not harm convergence because convergence requires only that $f(u^k)$ decreases by some minimal amount for each iteration $k$, and the adjustment operator can only make the decrease larger. 

\begin{algorithm}
\caption{Backtracking line search method with trial point adjustment}
\begin{algorithmic}[1]
	\REQUIRE $u^0 \in \reals^{N_u}$, $\delta \in (0,1/2)$, $\alpha \in (0,1)$ 	\FOR{$k = 0,1,2,\ldots$}
		\STATE $g = \nabla f(u^k)$, $B = \Bfun(u^k)$
		\STATE Solve for $\Delta u$: $B \Delta u = -g$
		\STATE $u_p(s) = u^k + s \Delta u$
		\STATE  Find the smallest $j \geq 0$ such that  $f(u_d(u_p(\alpha^j))) - f(u^k) \leq  \delta g^\Trans (u_p(\alpha^j) - u^k)$
		\STATE $u^{k+1} = u_d(u_p(\alpha^j))$
	\ENDFOR
\end{algorithmic}
\label{mod_newt_method_ls}
\end{algorithm}

To make Alg.~\ref{mod_newt_method_ls} into a semi-reduced method for minimizing a function $F(x) = F(y,z)$, we set $f(u) = F(x)$ and put system $B \Delta x = -g$ into the block Gaussian decomposed form \eqref{schur_system}. We then require the trial point adjustment to have the form $x_d(y,z) = \left( y, z_d(y,z) \right)$, so that only $z$ can change. The result of these changes is Alg.~\ref{semi_reduced_ls}.

\begin{algorithm}
\caption{Semi-reduced line search method.}
\begin{algorithmic}[1]
	\REQUIRE $x^0 = (y^0,z^0) \in \reals^N$, $\delta \in (0,1/2)$, $\alpha \in (0,1)$
	\STATE Define $x_d(y,z) = (y, z_d(y))$
	\FOR{$k = 0,1,2,\ldots$}
		\STATE $g = \nabla F(x^k),  B = \Bfun(x^k)$
		\STATE Solve for $\Delta y$: $B_s \Delta y = -g_y + B_{yz} B_{zz}^{-1} g_z$
		\STATE Solve for $\Delta z$: $B_{zz} \Delta z = -g_z - B_{zy} \Delta y$
		\STATE Define $x_p(s) = (y_p(s),z_p(s)) = (y^k + s \Delta y, z^k + s \Delta z)$
		\STATE  Find the smallest $j \geq 0$ such that 
			$F(x_d(x_p(\alpha^j))) - F(x^k) \leq  \delta g^\Trans (x_p(\alpha^j) - x^k)$
		\STATE $x^{k+1} = x_d(x_p(\alpha^j))$
	\ENDFOR
\end{algorithmic}
\label{semi_reduced_ls}
\end{algorithm}

To make Alg.~\ref{semi_reduced_ls} into a reduced update method, we assume our trial point adjustment is unique and optimal, $z_d(y,z) = z_m(y) = \argmin_z F(y,z)$, and exploit this fact to simplify the algorithm. Optimal adjustments ensure that $g_z = \nabla_z F(y^k,z_m(y^k)) = 0$ for all $k$, so terms involving $g_z$ disappear. In particular, line 7 reduces to $g^\Trans (x_p(\alpha^j) - x^k) = g_y^\Trans (y_p(\alpha^j) - y^k)$. After terms involving $g_z$ are removed, the trial point $z_p(s) = z^k + s \Delta z$ appears only within the expression $x_d(x_p(\alpha^j))$. But if we write out $x_d(x_p(s)) = (y^k + s\Delta y, z_m(y^k + s\Delta y))$, we see that $z^k + s \Delta z$ has been supplanted by the adjusted point $z_m(y^k + s\Delta y)$, so we may skip it by redefining $x_p$ as $x_p(s) = (y^k + s \Delta y, z_m(y^k + s \Delta y))$. The disappearance of $z^k + s \Delta z$ renders the step $\Delta z$ unused in any way, so line 5 can be deleted. What is left is Alg.~\ref{mod_ls_method_exact_inner}, a \emph{simplified semi-reduced method}. In the next section we show that, when $B$ is chosen appropriately, versions of this simplified semi-reduced method are identical to several reduced (variable elimination) methods in the literature.

\begin{algorithm}
\caption{Simplified semi-reduced line search method.}
\begin{algorithmic}[1]
	\REQUIRE $x^0 = (y^0,z_m(y^0)) \in \reals^N$, $\delta \in (0,1/2)$, $\alpha \in (0,1)$ 
	\FOR{$k = 0,1,2,\ldots$}
		\STATE $g_y = \nabla_y F(x^k),  B = \Bfun(x^k)$
		\STATE Solve $B_s \Delta y = -g_y$
		\STATE Define $x_p(s) = (y_p(s),z_p(s)) =  (y^k + s \Delta y, z_m(y^k + s \Delta y))$
		\STATE  Find the smallest $j \geq 0$ such that 
			$F(x_p(\alpha^j)) - F(x^k) \leq  \delta g_y^\Trans (y_p(\alpha^j) - y^k)$
		\STATE $x^{k+1} = x_p(\alpha^j)$
	\ENDFOR
\end{algorithmic}
\label{mod_ls_method_exact_inner}
\end{algorithm}

Reinterpreting variable elimination as a simplified semi-reduced method allows us to precisely articulate the cost-benefit tradeoff involved in using variable elimination, as well as the \emph{raison d'\^{e}tre} for \emph{non-}simplified semi-reduced methods. The benefit of variable elimination is that we need not compute $g_z$, $\Delta z$, or quantities dependent on them, and the trial point adjustments may cause the algorithm to converge faster. The cost is that we must compute the \emph{optimal} $z$ value after every $y$ update, while in semi-reduced updates we only require that the adjustment does not increase the objective. Variable elimination is preferable only if the adjustment subproblem can be solved quite quickly and yields a significantly increased convergence rate. While this condition often holds in unconstrained least squares problems, in general calculating $\argmin_z F(y,z)$ is often quite costly and may not be worth the trouble. Semi-reduced methods permit us to forgo this cost, granting increased flexibility without compromising convergence.

\subsection{Variable elimination as a simplified semi-reduced method} \label{varElimSemiRed}
Here we show that three popular reduced (variable elimination) methods can all be interpreted as simplified semi-reduced methods with an appropriate Hessian model. In other words, reduced methods can be obtained by operations on $F(y,z)$ alone, without ever forming the objective $F_r(y)$ explicitly. This surprising result is essentially due to the implicit function theorem and the fact that optimization methods only use very limited local information about a function to determine iterates. We begin with a new lemma stating the exact condition required for a reduced and a simplified semi-reduced method to be equivalent. 

\begin{lemma} \label{nd_newton_is_reduced_newton}
Let $y^0 \in \reals^{N_y}$ be given, and let $z^0 = z_m(y^0)$. Let invertible Hessian models $\Bfun_r(y)$ and $\Bfun_f(y,z)$ for $F_r(y)$ and $F(y,z)$ be given. Assume that $z_m(y)$ is well-defined: that is, there is a unique solution of $\min_z F(y,z)$ for any given $y$. Consider the following pair of Newton-type algorithms:
\begin{enumerate}
\item \emph{Reduced method:}
Alg.~\ref{mod_newt_method_ls} with $f(u) = F_r(y)$, $y_d(y) = y$, $\Bfun = \Bfun_r$.
\item \emph{Simplified semi-reduced method:}
Alg.~\ref{mod_ls_method_exact_inner} with $\Bfun = \Bfun_f$.
\end{enumerate}
Let $B_s = B_{yy} - B_{yz} B_{zz}^{-1} B_{zy}$. These two algorithms generate identical iterates if and only if, at all points $y^k$ visited by each algorithm, the Hessian models $B_r = \Bfun_r(y)$ and $B = \Bfun_f(y,z_m(y))$ obey
\beq
	B_r = B_s.
	\label{schur_rel}
\eeq
\end{lemma}
\begin{proof}
After the specified substitutions are made, Algs.~\ref{mod_newt_method_ls} and \ref{mod_ls_method_exact_inner} have exactly one difference: the gradient used in Alg.~\ref{mod_newt_method_ls} is $\nabla F_r(y)$, while in Alg.~\ref{mod_ls_method_exact_inner} it is $\nabla_y F(y,z)$. Thus it suffices to show that $\nabla F_r(y) = \nabla_y F(y,z)$. Letting $Dz_m$ denote the Jacobian of $z_m(y)$, we have
\begin{equation} 
	\nabla F_r(y) = \nabla_y F(y,z_m(y)) + Dz_m \cdot \nabla_z F(y,z_m(y)) = \nabla_y F(y,z_m(y)),
	\label{reduced_gradient}
\end{equation}
where the second term has vanished because $z_m(y)$ is a stationary point of $F(y,z)$, so $\nabla_z F(y,z_m(y)) = 0$. 
\end{proof}

Now we show that the reduced Newton method (i.e.~Newton's method on $F_r(y)$) can be interpreted as a simplified semi-reduced Newton method on $F(y,z)$. This was implicitly shown by Richards \cite{richards1961mle} for the classical, nonglobalized Newton iteration. 

\begin{proposition} \label{red_newt_is_mod_newt}
Under the assumptions of Lemma \ref{nd_newton_is_reduced_newton}, the reduced method Alg.~\ref{mod_newt_method_ls} with Hessian model $B_r = \nabla^2 F_r$ is equivalent to the simplified semi-reduced method Alg.~\ref{mod_ls_method_exact_inner} with model $B = \nabla^2 F$.

\end{proposition}
\begin{proof}
We need only verify the Schur complement relation \eqref{schur_rel}. Differentiating \eqref{reduced_gradient}, we have
\begin{equation} 
	\nabla^2 F_r = \nabla^2_{yy} F + \nabla^2_{yz} F \cdot Dz_m.
	\label{reduced_hessian}
\end{equation}
$Dz_m$ can be obtained by implicit differentiation of the stationary point condition $\nabla_z F(y,z_m(y)) = 0$:
\begin{align}
	&\nabla^2_{zy} F(y,z_m(y)) + \nabla^2_{zz} F(y,z_m(y)) \cdot Dz_m = 0 \\
	&Dz_m = - [\nabla^2_{zz} F]^{-1} \nabla^2_{zy} F.
	 \label{jacobiaN_zm}
\end{align}
Plugging this expression into \eqref{reduced_hessian} and setting $B_r = \nabla^2 F_r$ and $B = \nabla^2 F$ yields \eqref{schur_rel} as desired.
\end{proof}

Now consider the separable case, where $F(y,z) = L(A(y)z)$, but $L(\mu)$ is not necessarily a least squares functional. We derive two simplified semi-reduced methods for this objective. In the least squares case, these methods are equivalent to the Kaufman \cite{Kaufman:1975} and Golub-Pereyra \cite{golub1973differentiation,Golub:separablenonlinear:2002} variants of variable projection, but they also apply to general nonquadratic $L$, a case for which no reduced method existed before.
To derive our methods, we note that the variable projection model Hessians $B_r$ have a \emph{closed-form normal decomposition}: they can be written as $B_r = X_r^\Trans X_r$ for some explicit $X_r$. Accordingly we will seek Hessian models $B$ such that $B_s = X_s^\Trans X_s$ for some closed-form $X_s$. 

We set some notation and conventions before we begin. Let $X_{:,j}$ the $j^{th}$ column of a matrix $X$. For any full column rank matrix $X$, $X^\pinv = (X^\Trans X)^{-1} X^\Trans$ is the Moore-Penrose psuedoinverse and $P_{X}^\perp = I - X X^\pinv$ is the orthogonal projector onto $\text{range}(X)^\perp$. Given a function $f(u,v)$ let $Df = [\partial_u f, \partial_v f]$ denote its Jacobian. To simplify our formulas we define the quantities $\mu(y,z) = A(y)z$, $W = (\nabla^2 L)_\mu^{1/2}$, and $\bar{A} = WA$. We abuse notation by ignoring the implicit dependence of $W$ on $y$ and $z$, which allows us to write $W \partial_{y_j} A$ as $\partial_{y_j} \bar{A}$.

We begin by decomposing the full Hessian of $F$ into two components: $\nabla^2 F = G + E$. The $G$ term is the Gauss-Newton Hessian model, $G = J^\Trans J$, where $J = W(D\mu)$. The blocks of $J$ are given by
\begin{equation}
	 (J_y)_{:,j} = (\partial_{y_j} \bar{A}) z  \,\,\, \text{for $j = 1,\ldots,N_y$}, \qquad J_z = \bar{A}.
	\label{sep_jac}
\end{equation}
The $E$ component is a residual term given by $E = \sum_i (\nabla L)_i \nabla^2 \mu_i$. Note that $E_{zz} = 0$ because $\nabla_{zz}^2 \mu_i = 0$ for all $i$. 

The first Hessian model we consider will be $G$. A closed-form normal decomposition for $G_s$ can be derived by:
\begin{equation}
	G_s 
		= J_y^\Trans (I - \bar{A} \bar{A}^\pinv ) J_y 
		= J_y^\Trans P_{\bar{A}}^\perp J_y  = (-P_{\bar{A}}^\perp J_y)^\Trans (-P_{\bar{A}}^\perp J_y) = J_s^\Trans J_s,
	\label{gn_schur}
\end{equation}
where the last line uses the fact that orthogonal projection is symmetric and idempotent, and the minus sign has been introduced for consistency with the variable projection convention. By Lemma \ref{nd_newton_is_reduced_newton} this result yields a pair of equivalent reduced and simplified semi-reduced methods for any $L(\mu)$:
\begin{proposition} The reduced method Alg.~\ref{mod_newt_method_ls} with Hessian model $B_r = G_s$ is equivalent to the simplified semi-reduced method Alg.~\ref{mod_ls_method_exact_inner} with model $B = G$.
\label{kf_gen_equiv}
\end{proposition}
\noindent In the least squares case we have $z = z_m(y) = A^\pinv b$, so $J_s = -P_{A}^\perp (\partial_{y_j} A) z = -P_{A}^\perp (\partial_{y_j} A) A^\pinv b$, and this $J_s$ is precisely the reduced Jacobian $J_r$ proposed by Kaufman. Thus we have $G_s = G_r$ and the following result, which was proven by Ruhe and Wedin in \cite{ruhe1980algorithms} for algorithms without globalization:
\begin{corollary}
 Kaufman's variable projection method is equivalent to a simplified semi-reduced method for separable least squares using $B = G$.
\label{kf_lsq_equiv}
\end{corollary}

Next we express the Golub-Pereyra variable projection method as a simplified semi-reduced method. To do this we need a Hessian model $H$ such that $H_s = K_s^\Trans K_s$, where $K_s$ is equal to the Golub-Pereyra reduced Jacobian $K_r$. This is a challenging problem because the Golub-Pereyra model $H_r = K_r^\Trans K_r$ is a closer approximation to $\nabla^2 F_r$ than the Kaufman model $G_r$, but there is no obvious normally decomposable $H$ that approximates $\nabla^2 F$ better than the traditional Gauss-Newton model $G$. 

Fortunately the model may be derived by an ingenious technique due to Ruhe and Wedin. Essentially, their idea is to apply a block Cholesky factorization to $\nabla^2 F$ and use the factors to help reduce the discrepancy between $G$ and $\nabla^2 F$. In our notation the Cholesky factorization used is the $UDU^\Trans$ factorization, which is simply the more familiar $LDL^\Trans$ factorization \cite{NocedalWright:NumericalOptimization:2006} with the conventional variable order reversed. Given a matrix $X$, we write its $UDU^\Trans$ factorization as $X = U \hat{X} U^\Trans$, where
\beq
	\hat{X} = 
	\begin{bmatrix}
		X_s	& 0		\\
		0	& X_{zz}	\\
	\end{bmatrix},
	\quad
	U = 
	\begin{bmatrix}
		I 	& X_{yz} X_{zz}^{-1} 	\\
		0	& I 				\\
	\end{bmatrix}
	\label{udu}
\eeq
and $X_s = X_{yy} - X_{yz} X_{zz}^{-1} X_{zy}$. Note that the $U$ factor is determined uniquely by its $yz$ block. Setting $X = \nabla^2 F$ we have $U_{yz} = (G_{yz} + E_{yz})G_{zz}^{-1}$.

To derive $H$, consider the product $U^{-1} G U^{-\Trans}$, which is positive definite and normally decomposable because $G$ is. If $G$ were the true Hessian $U^{-1} G U^{-\Trans}$ would be block diagonal, but in reality
\beq
	U^{-1} G U^{-\Trans} = 
	\begin{bmatrix}
		G_s + E_{yz} G_{zz}^{-1} E_{zy} 	&	-E_{yz} \\
		-E_{zy}						&	G_{zz}
	\end{bmatrix}.
	\label{ugu}
\eeq
Letting $\hat{H}$ denote the diagonal of $U^{-1} G U^{-\Trans}$, we can define a new positive definite and normally decomposable Hessian model by setting $H \triangleq U \hat{H} U^\Trans$. From \eqref{udu} it immediately follows that
\beq
	H_s = \hat{H}_{yy} = G_s + E_{yz} G_{zz}^{-1} E_{zy}.
\eeq

Now we express $H_s$ in the form $H_s = K_s^\Trans K_s$. We have already decomposed $G_s = J_s^\Trans J_s$; a similar formula for the second term, $E_{yz} G_{zz}^{-1} E_{zy}$, is given by
\beq
	E_{yz} G_{zz}^{-1} E_{zy} 
		= E_{yz} (\bar{A}^\Trans \bar{A})^{-1} E_{zy} 
		= [(\bar{A}^\pinv)^\Trans E_{zy}]^\Trans [(\bar{A}^\pinv)^\Trans E_{zy}] = M^\Trans M,
\eeq
where we have used the identity $(X^\Trans X)^{-1} = X^\pinv (X^\pinv)^\Trans$ valid for any matrix $X$ with full column rank. We now have $H_s = J_s^\Trans J_s + M^\Trans M$, where $J_s = -P_{\bar{A}}^\perp J_y$ and $M = (\bar{A}^\pinv)^\Trans E_{zy}$. Surprisingly we may rewrite this as $H_s = (J_s + M)^\Trans (J_s+M)$ because the cross terms vanish: $J_s^\Trans M = -J_y^\Trans P_{\bar{A}}^\perp (\bar{A}^\pinv)^\Trans E_{zy} = 0$ for $P_{X}^\perp (X^\pinv)^\Trans = 0$. Therefore, by setting $K_s = J_s + M$, we have $H_s = K_s^\Trans K_s$ as desired.

All that remains is to compute $K_s$, which we do column-by-column. The $j^{th}$ column of $J_s$ is
	$(J_s)_{:,j} = (-P_{\bar{A}}^\perp J_y)_{:,j} = -P_{\bar{A}}^\perp (\partial_{y_j} \bar{A}) z,$ 
while the $j^{th}$ column of $E_{zy}$ is given elementwise by
\begin{equation}
	(E_{zy})_{kj}
		= \sum_i (\nabla L)_i \partial_{z_k} \partial_{y_j} (A z)_i 
		= \sum_i (\nabla L)_i (\partial_{y_k} A)_{ik} 
		= [(\partial_{y_j} A)^\Trans \nabla L]_k,
\end{equation}
so we have $M_{:,j} = (\bar{A}^\pinv)^\Trans(\partial_{y_j} A)^\Trans \nabla L$. We write this in terms of $\bar{A}$ by defining the weighted residual $r = W^{-1} \nabla L$, so that $M_{:,j} = (\bar{A}^\pinv)^\Trans(\partial_{y_j} \bar{A})^\Trans r$. Thus the desired formula for $K_s$'s columns is
\beq
	(K_s)_{:,j} = (J_s)_{:,j} + M_{:,j} =  -P_{\bar{A}}^\perp (\partial_{y_j} \bar{A}) z + (\bar{A}^\pinv)^\Trans(\partial_{y_j} \bar{A})^\Trans r.
	\label{gp_jacobian}
\eeq
Again invoking Lemma \ref{nd_newton_is_reduced_newton}, we have shown that
\begin{proposition} The reduced method Alg.~\ref{mod_newt_method_ls} with Hessian model $B_r = H_s$ is equivalent to the simplified semi-reduced method Alg.~\ref{mod_ls_method_exact_inner} with model $B = H$.
\label{gp_gen_equiv}
\end{proposition}

Specializing this result to the least-squares case $L(\mu) = \frac{1}{2} \norm{\mu - b}^2$ as before, we have $r = Az - b = AA^\pinv b - b = - P_{A}^\perp b$, and $(K_s)_{:,j}$ simplifies to
\beq
	(K_s)_{:,j} = -\left(P_{A}^\perp (\partial_{y_j} A) A^\pinv + (P_{A}^\perp (\partial_{y_j} A) A^\pinv)^\Trans \right) b,
\eeq
which is precisely the Jacobian $K_r$ of the reduced functional $F(y,z_m(y)) = \frac{1}{2} \norm{-P_A^\perp b}^2$ derived by Golub and Pereyra \cite{Golub:separablenonlinear:2002}. Since $K_r = K_s$, we have $H_r = H_s$ and the desired equivalence:
\begin{corollary}
The Golub-Pereyra variable projection method is equivalent to a simplified semi-reduced method for separable least squares using Hessian model $H$. 
\label{gp_lsq_equiv}
\end{corollary}

\subsection{Semi-reduced methods as the natural generalization of variable elimination}
Proposition \ref{red_newt_is_mod_newt} and Corollaries \ref{kf_lsq_equiv} and \ref{gp_lsq_equiv} show that the reduced Newton's method and both variants of variable projection can be interpreted as simplified semi-reduced methods. In addition, Propositions \ref{kf_gen_equiv} and \ref{gp_gen_equiv} define new simplified semi-reduced methods that generalize variable projection to nonquadratic $L(\mu)$.

Unfortunately these algorithms are of more theoretical than practical use, for the following reasons. First, we still have not dealt with the problem of computing $z_m(y)$. In general there is no closed form for $z_m(y)$, and computing it may be so expensive that the computational burden outweighs any increase in convergence rate over a simple full or alternating update method. Second, if the domain of $\ell_i(\mu_i)$ is a bounded subset of $\reals$, as is true for the Poisson and several other log-likelihoods, the bounds often must be enforced via reparametrization or constrained optimization. This adds still more complexity and in the latter case makes unconstrained optimization inapplicable. 

The driving technical insight of this paper is the following: if we forgo the simplifications afforded by using optimal block trial point adjustment and use an ordinary semi-reduced method instead, \emph{all of these barriers and difficulties disappear.} Trial point adjustments need not be optimal, so there is no need for the computationally expensive $z_m(y)$, and constraints can be handled by incorporating trial point adjustment and block Gaussian elimination into classical full update methods. Thus, semi-reduced methods provide a natural way to extend variable elimination methods beyond least squares.

\section{A semi-reduced method for bound constrained and nonquadratic problems} \label{bound_constrained_opt}
In this section we present a classical method for smooth bound-constrained problems and turn it into a semi-reduced method. The problem we wish to solve is
\begin{equation}
	\minimize f(x) \quad \subto l \leq x \leq u,
	\label{min_bound}
\end{equation}
where $-\infty \leq l \leq u \leq \infty$ are vectors bounding the components of $x \in \reals^N$, and $f(x)$ is twice differentiable. The method we present is a trial point adjusted variant of Bertsekas's projected Newton method \cite{bertsekas1982projected, gafni1984two}. (In our terminology Bertsekas's method is better described as a \emph{projected Newton-type method}, since it allows for approximate Hessian models.) We choose this method because it is relatively simple, its convergence is global and potentially superlinear, and similar second-order gradient projection methods are empirically among the state-of-the-art for a variety of constrained inverse problems \cite{bardsley2004nonnegatively,schmidt2007fast,schmidt20111,vogelComputationalMethods}.  

The update in the projected Newton-type method is of the form
\beq
	x^{k+1} = \Proj(x^k - S^k \nabla f(x^k)),
\eeq
where $S^k$ is a scaling matrix which we will assume to be positive definite, and $\Proj(w)$ is the projection of $w$ onto the box $\boxSet = \{w \sepr l \leq w \leq u\}$ given componentwise by
\begin{equation}
	\Proj(w)_i \triangleq \text{median}(l_i,w_i,u_i).
\end{equation}
This iteration is a generalization of the projected gradient method, which restricts $S^k$ to be a multiple of the identity. Bertsekas showed that the naive Newton-type choice $S^k = (B^k)^{-1}$ with $B^k \approx \nabla^2 f(x^k)$ can cause convergence failures, but convergence can be assured by modifying the naive choice using a very simple active set strategy, in which the Hessian is modified to be diagonal with respect to the active indices.

To describe projected Newton-type methods we will use the following notation. Let $[N] \triangleq \{1,\ldots,N\}$, and for any $J \subset [N]$, let $v_J = (v_i)_{i \in J}$ denote the subvector of $v \in \reals^N$ indexed by $J$ and $X_{J,J} = [X_{ij}]_{i,j \in J}$ the indexed submatrix of $X \in \reals^{n \times n}$. Given $\epsilon \geq 0$ and $x \in \reals^n$, we define the active set associated with $x$ by
\begin{equation}
	\Active(x) = 
	\{ i \in [N]  \sepr (\nabla f(x)_i > 0, \,\, x_i \leq l_i + \epsilon)  
		\,\,\text{or}\,\, (\nabla f(x)_i < 0, \,\, x_i \geq u_i - \epsilon) \},
\end{equation}
and the inactive set as its complement, $\Inactive(x) = [N] - \Active(x)$. Using this notation we present in Alg.~\ref{alg:tmp} the projected Newton-type method with trial point adjustment, where an identity scaling matrix is chosen on the active set for concreteness. As in the unconstrained case, the only difference between the semi-reduced method Alg.~\ref{alg:tmp} and the original full update method (found in equations (32)\textendash(37) of \cite{bertsekas1982projected}) is the addition of the adjustment operator $x_d(x)$: specifically, on the left hand side of line 5 and the right hand side of line 6 of Alg.~\ref{alg:tmp}, our method has $x_d(x_p(\alpha^j))$ while \cite{bertsekas1982projected} has only $x_p(\alpha^j)$. Careful examination of the global convergence proof, Proposition 2 in \cite{bertsekas1982projected}, reveals that, with very minor additions, it also establishes convergence of Alg.~\ref{alg:tmp}. Here we review the argument very briefly, with just enough detail to describe how to adapt it to accommodate trial point adjustment. 

\begin{algorithm}
\caption{Projected Newton-type method with trial point adjustment.}
\begin{algorithmic}[1]
	\REQUIRE $x^0 \in \reals^{N}$, $\delta \in (0,1/2)$, $\alpha \in (0,1)$ 	\FOR{$k = 0,1,2,\ldots$}
		\STATE $g = \nabla f(x^k)$, $B^k = \Bfun(x^k)$
		\STATE $\Delta x_I = -(B_{I,I}^k)^{-1} g_I$, $\Delta x_A = -g_A$
		\STATE Define $x_p(s) = \Proj(x^k + s \Delta x)$
		\STATE  Find the smallest $j \geq 0$ such that  			
		\beq
		f(x_d(x_p(\alpha^j)) - f(x^k) \leq \delta \left\{  g_I^\Trans(\alpha^j \Delta x_I) + 
			g_A^\Trans (x_p(\alpha^j)_A - x^k_A)  \right\}
		\label{armijo_bert}
		\eeq
		\STATE $x^{k+1} = x_d(x_p(\alpha^j))$
		\STATE $\epsilon \str \min(\epsilon_0, \norm{x_p(1) - x^k})$
	\ENDFOR
\end{algorithmic}
\label{alg:tmp}
\end{algorithm}

\begin{proposition}
Assume that $\nabla f(x)$ is Lipschitz continuous on any bounded set of $\reals^N$ and the eigenvalues of $B^k$ are uniformly bounded away from zero and infinity for all $k$. Then every limit point of Alg.~\ref{alg:tmp} is a stationary point of Problem \eqref{min_bound}.
\begin{proof}

By contradiction: suppose a subsequence $(x^k)_{k \in K}$ of $(x^k)_{k = 0}^\infty$ exists such that $\lim_{k \rightarrow \infty, k \in K} x^k = \bar{x}$, where $\bar{x}$ is not a critical point. Let $s_k = \alpha^{j_k}$ denote the step size chosen at iteration $k$ on line 5; the proof of Proposition 2 of \cite{bertsekas1982projected} first shows that the monotonicity of the sequence $(f(x^k))_{k = 0}^\infty$, Lipschitz continuity of $\nabla f(x)$, eigenvalue bound on $B^k$, and the nonpositivity of the terms on the right hand side of \eqref{armijo_bert} together imply that $\liminf_{k \rightarrow \infty, k \in K} s_k = 0$. Since all the required properties still hold in our case, this conclusion holds for Alg.~\ref{alg:tmp} as well. Next it is shown that, for some $\bar{s} > 0$ independent of $k$, we have
\beq
	f(x_p(s)) - f(x_k) \leq \delta \{ g_I^\Trans(s\Delta x_I) + 
			g_A^\Trans (x_p(s)_A - x^k_A)\} \quad \text{for $s \leq \bar{s}$.}
\eeq
The computation supporting this claim depends only on the properties of $f(x), x_p(s), A, \text{and}\,\, I$, so it still holds for Alg.~\ref{alg:tmp}. But $f(x_d(x)) \leq f(x)$ for all $x$, so
\beq
	f(x_d(x_p(s))) - f(x_k) \leq \delta \{ g_I^\Trans(s\Delta x_I) + 
			g_A^\Trans (x_p(s)_A - x^k_A)\} \quad \text{for $s \leq \bar{s}$.}
\eeq
It follows that $s_k \geq \alpha^J$ where $J$ is the smallest nonnegative integer such that $\alpha^J \leq \bar{s}$, contradicting $\liminf_{k \rightarrow \infty, k \in K} s_k = 0$.
\end{proof}
\end{proposition}
\noindent Careful examination of the proofs in \cite{bertsekas1982projected} indicates that the other properties of the projected Newton-type method generally continue to hold for the trial point adjusted version, but the full details are beyond the scope of this paper.

Since $B$ is only required to have eigenvalues bounded away from $0$ and $\infty$, Alg.~\ref{alg:tmp} can accommodate a wide variety of Hessian models and regularization strategies. In our numerical experiments we use Alg.~\ref{alg:tmp_damp}, which is a special case of Alg.~\ref{alg:tmp} and thus inherits its convergence properties. Alg.~\ref{alg:tmp_damp} sets $B = \Bfun(x) + \lambda I$, where $\Bfun(x)$ is a Gauss-Newton Hessian, $\lambda I$ is a Levenberg-Marquardt damping term, and $\lambda$ is adjusted at every iteration according to a step quality metric $\rho$. Levenberg-Marquardt regularization is useful for guarding against rank-deficient Hessian models \cite{NocedalWright:NumericalOptimization:2006}.

\begin{algorithm} 
\caption{Damped projected Newton-type method with trial point adjustment.}
\begin{algorithmic}[1]
	\REQUIRE $\delta \in (0,1/2)$, $\alpha \in (0,1)$,  $\lambda_{min}, \lambda_{max} \in [0,\infty)$, $0 \leq \rho_{bad} < \rho_{good} \leq 1$
	\REQUIRE $\tau \in (0, \infty)$, $\epsilon_0 \in (0,\infty)$
	\STATE $\epsilon \str \epsilon_0$
	\FOR{$k = 0,1,2,\ldots$}
		\STATE $B = \Bfun(x^k) + \lambda I$, $g = \nabla f(x^k)$, $A = \Active(x^k)$, $I = \Inactive(x^k)$
		\STATE If $\norm{\Proj(x^k - g) - x} \leq \tau$ or $k \geq k_{max}$, stop.
		\STATE Solve $B_{I,I} \Delta x_I = -g_I$. Set $\Delta x_A = -g_A$.
		\STATE Define $x_p(s) = \Proj(x^k + s \Delta x)$
		\STATE  Find the smallest $j \geq 0$ such that 
			\[f(x_d(x_p(\alpha^j)) - f(x^k) \leq \delta \left\{  g_I^\Trans(\alpha^j \Delta x_I) + g_A^\Trans (x_p(\alpha^j)_A - x^k_A)  \right\}\]
		\STATE $x^{k+1} = x_d(x_p(\alpha^j))$
		\STATE $\rho = (f(x^k)-f(x^{k+1}))/(-\tfrac{1}{2}g_I^\Trans \Delta x_I)$
		\IF{$\rho > \rho_{good}$}
			\STATE $\lambda \str \max(\lambda / 2, \lambda_{min})$
		\ELSIF{$\rho < \rho_{bad}$}
			\STATE $\lambda \str \min(10 \lambda, \lambda_{max})$
		\ENDIF
		\STATE $\epsilon \str \min(\epsilon_0, \norm{x_p(1) - x^k})$
	\ENDFOR
\end{algorithmic}
\label{alg:tmp_damp}
\end{algorithm}

While any linear algebra technique may be used to solve the Newton-type systems in Algs.~\ref{alg:tmp} and \ref{alg:tmp_damp}, block Gaussian elimination is of particular interest because of the role it plays in our semi-reduced framework. The block Gaussian elimination methods used in our numerical experiments are introduced in the next section.

\section{Using block Gaussian elimination to exploit separable structure} \label{sec:block_lin_algs}
One of the key advantages of variable elimination methods is their ability to take advantage of special problem structure, such as multiple measurement vectors \cite{kaufman1992separable}. Block Gaussian elimination can be used to derive linear solvers with similar structure exploiting properties. Here we describe two such algorithms which we claim can provide an advantage over standard methods; these claims are tested in our experiments below. 

The first method is a QR method for normal equations, and is thus appropriate for methods employing a Gauss-Newton Hessian model. This approach is most suited for highly ill-conditioned systems, such as those arising from exponential fitting and other difficult problems traditionally tackled by variable projection. Generalized Gauss-Newton Hessian models for nonquadratic likelihoods can be handled by this method \cite{Bunch:quasiLikeAlg:1993}. The second method is for problems where $z$ is very high dimensional (a vectorized image or volume array for example), while $y$ is relatively low-dimensional. It is similar to the linear algebra algorithms used in the reduced update optimizers defined in \cite{Vogel:PhaseDiv:1998,Chung:2010}. Unlike these algorithms, which are designed for specific least squares optimization tasks, our algorithm can be used in any Newton-type optimizer, including ones that handle Poisson likelihoods or bound constraints.

\subsection{Solving normal equations by block decomposed QR factorization}
We now present a method for solving normal equations by block Gaussian elimination and QR factorization. Normal equations are systems of the form
\beq
	J^\Trans J \Delta x = -J^\Trans r,
	\label{normal_eq}
\eeq
where $J \in \reals^{m \times N}$. The Newton-type system $B \Delta x = -g$ has this form when we use a Gauss-Newton Hessian model or its generalization for non-quadratic likelihoods \cite{Bunch:quasiLikeAlg:1993}. Assuming $B = J^\Trans J$, the reduced and damped Gauss-Newton system $(B_{I,I} + \lambda I) \Delta x_I = -g_I$ from Alg.~\ref{alg:tmp_damp} can also be written in this form by deleting columns from $J$ and augmenting the result with the scaled identity matrix $\sqrt{\lambda} I$ \cite{NocedalWright:NumericalOptimization:2006}.

Cholesky factorization is the fastest way to solve normal equations, but rounding error can amplify to unacceptable levels when $J$ is highly ill-conditioned, as in some curve fitting problems. Greater accuracy can be gained at the expense of additional computation by QR factorizing $J$. Assuming $J$ is full rank, we will write the (thin) QR factorization as $[Q,R] = \texttt{qr}(J)$, where $Q \in \reals^{m \times N}$ is an orthogonal matrix and $R \in \reals^{N \times N}$ is an invertible upper triangular matrix. Substituting $J = QR$ into \eqref{normal_eq} and noting that $Q^\Trans Q = I$, we obtain the solution
\beq
	\Delta x = - R^{-1} Q^\Trans r.
\eeq
In our method we solve \eqref{normal_eq} by QR factorizing not the system itself, but its block decomposed form \eqref{schur_system}. We begin by putting \eqref{schur_system} in normal equation form. From \eqref{gn_schur} we have $B_s = J_s^\Trans J_s$, where $J_s = P_{J_z}^\perp J_y$. Similarly we have
$	-g_y + B_{yz} B_{zz}^{-1} g_z = - J_s^\Trans r.$
From these results we can write \eqref{schur_system} as a pair of normal equations:
\begin{subequations}
\begin{align}
	(J_s^\Trans J_s) \Delta y
	&= -J_s^\Trans r \label{y_schur_eqn_lsq3} \\
	(J_z^\Trans J_z) \Delta z
	&= - J_z^\Trans (r + J_y \Delta y) \label{z_schur_eqn_lsq3}.
\end{align} \label{schur_system_lsq3}%
\end{subequations}
To compute $J_s$, we need to compute $P_{J_z}^\perp$. This may be done using the QR factorization of $J_z$: if $X = QR$ is the QR factorization of a matrix $X$ with full column rank, we have
\beq
P_X^\perp = I - Q Q^\Trans.
\eeq
Using this result we can form $J_s$ and solve the system as described in Alg.~\ref{qr_decomp_solve}.
\begin{algorithm}
\caption{Solution of $J^\Trans J \Delta x = -J^\Trans r$ by block decomposed QR.}
\begin{algorithmic}[1]
	\STATE $[Q_z, R_z] = \texttt{qr}(J_z) $
	\STATE $[t,T] = Q_z^\Trans [r, J_y]$
	\STATE $J_s = J_y - Q_z T$
	\STATE $[Q_s, R_s] = \texttt{qr}(J_s) $
	\STATE $\Delta y = - R_s^{-1} Q_s^\Trans r$
	\STATE $\Delta z = - R_z^{-1} (t + T \Delta y)$
\end{algorithmic}
\label{qr_decomp_solve}%
\end{algorithm}

Alg.~\ref{qr_decomp_solve} is useful when $J_z$ has structure that makes its QR factorization easier to compute than that of the full $J$. As an example, suppose that $J_z$ is a block diagonal matrix with blocks $J_z^{(i)}$ for $i = 1,\ldots,n$. Such matrices arise in separable problems with multiple measurement vectors. In this case $Q_z$ and $R_z$ are block diagonal and Alg.~\ref{qr_decomp_solve} can be adapted to exploit this, as shown in Alg.~\ref{qr_decomp_solve_block}. Note that this algorithm never generates the large sparse matrix $J$, but only the nonzero blocks $J_y^{(i)}$ and $J_z^{(i)}$, which are computed just when they are needed. We expect this resource economy to result in reduced memory usage, higher cache efficiency, and ultimately a faster solution. 

\begin{algorithm}
\caption{Alg.~\ref{qr_decomp_solve} specialized to the case of block diagonal $J_z$.}
\begin{algorithmic}[1]
	\FOR{$i = 1,\ldots,n$}
		\STATE Compute $J_y^{(i)}, J_z^{(i)}$
		\STATE $[Q_z^{(i)}, R_z^{(i)}] = \texttt{qr}(J_z^{(i)})$
		\STATE $[t^{(i)},T^{(i)}] = [Q_z^{(i)}]^\Trans [r, J_y]$
		\STATE $J_s^{(i)} = J_y^{(i)} - Q_z^{(i)} T^{(i)}$
	\ENDFOR
	\STATE $[Q_s, R_s] = \texttt{qr}(J_s) $
	\STATE $\Delta y = - R_s^{-1} Q_s^\Trans r$
	\FOR{$i = 1,\ldots,n$}
		\STATE $\Delta z^{(i)} = - [R_z^{(i)}]^{-1} (t^{(i)} + T^{(i)} \Delta y)$
	\ENDFOR
\end{algorithmic}
\label{qr_decomp_solve_block}%
\end{algorithm}

\subsection{Mixed CG/Direct method for systems with one very large block.}
In some separable inverse problems, the number of linear variables $z$ is too large for direct solution by Cholesky or QR factorization. This is particularly true in image and volume reconstruction problems: if each pixel of a $256 \times 256$ pixel image is considered a free variable, which is very modest by imaging system standards, the relevant Jacobians and Hessians will be $65536 \times 63356$ and usually impossible to factorize or even store in memory. In this case conjugate gradients (CG) or other iterative linear algebra methods must be employed to solve the Newton-type systems $B \Delta x = -g$. These methods only need functions that compute matrix-vector products with $B$, which may be much less memory consuming if $B$ has special structure. Unfortunately the matrix $B$ is often ill-conditioned, which can lead to slow convergence of CG. In some cases, $B_{zz}$ is well conditioned, but the additional blocks involving the nonlinear variables $y$ result in a poorly conditioned $B$. A method that uses iterative linear algebra only on the subblock $B_{zz}$ has the potential to be more efficient.

Such a method may be derived by solving $B \Delta x = -g$ in the block decomposed form \eqref{schur_system}. We first solve \eqref{y_schur_eqn} by forming the small matrix $B_s = B_{yy} - B_{yz} B_{zz}^{-1}  B_{zy}$ column-by-column. We solve for $\Delta y$ by Cholesky factorizing this matrix, then solve \eqref{z_schur_eqn} by CG to obtain $\Delta z$, as summarized in Algorithm \ref{alg:mixed_dir_cg}. To understand when  Alg.~\ref{alg:mixed_dir_cg} may be more efficient than full CG, we roughly estimate and compare the costs of each algorithm. Let $t$ be the total floating point operations (flops) required to compute a matrix-vector product with $B$. We split $t$ into $t = t_y + t_z$, where $t_z$ is the cost of a matrix-vector product with $B_{zz}$, and $t_y$ is the cost of computing products with all three remaining blocks $B_{yy}$, $B_{yz}$, and $B_{zy}$. Then solving $B \Delta x = -g$ requires $T_{cg} = k (t_y + t_z)$ flops, where $k$ is the number of iterations required to achieve some suitable accuracy. 

\begin{algorithm}
\caption{Mixed CG/Direct solution of $B \Delta x = -g$.}
\begin{algorithmic}[1]
	\REQUIRE Functions that compute matrix-vector products with $B_{yy}$, $B_{yz}$, $B_{zy}$, $B_{zz}$. 
	Inverse matrix-vector products $B_{zz}^{-1} w$ are computed by conjugate gradients.
	\FOR{$i = 1,\ldots,N_y$}
		\STATE $(B_s)_{:,i} = B_s e_i$
	\ENDFOR
	\STATE Calculate a Cholesky factorization $R^\Trans R = B_s$
	\STATE $g_r = g_y - B_{yz} B_{zz}^{-1} g_z$	
	\STATE $ \Delta y = R^{-1} R^{-\Trans} g_r$
	\STATE $\Delta z = - B_{zz}^{-1} (g_z - B_{zy} \Delta y)$
\end{algorithmic}
\label{alg:mixed_dir_cg}
\end{algorithm}

In Alg.~\ref{alg:mixed_dir_cg}, we assume that computing $B_s$ is the dominant cost and the other computations are negligible, which is reasonable if $N_y$ is significantly greater than $1$. If $k_z$ is the number of CG iterations required to solve $B_{zz} u = w$ to suitable accuracy, then the cost of computing each column of $B_s$ is $t_y + k_z t_z$, yielding a total cost of $T_{mix} = N_y (t_y + k_z t_z)$ for all $N_y$ columns. By setting $T_{mix} \leq T_{cg}$, we see that Alg.~\ref{alg:mixed_dir_cg} will outperform full CG when the iterations $k$ required by full CG exceeds a certain threshold:
\beq
	k \gtrsim N_y \frac{t_y + k_z t_z}{t_y + t_z}.
	\label{mixed_beats_cg}
\eeq
The right-hand side is smallest when $t_y$ is much larger than $t_z$, $k_z$, and $N_y$ is relatively small; this corresponds to the case where $B_{zz}$ is relatively well conditioned, products with $B_{zz}$ are cheap, and there are not too many parameters in $y$. If $t_y \gg t_z$, then the threshold becomes $k \gtrsim N_y$. This is the minimum number of iterations we would expect from full CG if the eigenvalues of $B_{yy}$ are isolated, so Alg.~\ref{alg:mixed_dir_cg} should perform at least as well as full CG in this limit. However, if the spectrum of $B_{zz}$ and the other blocks combines unfavorably, the required iterations $k$ could be much larger, in which case Alg.~\ref{alg:mixed_dir_cg} should be more efficient.

Even when \eqref{mixed_beats_cg} does not hold, Alg.~\ref{alg:mixed_dir_cg} may still be desirable for other reasons. For example, if $B$ is much more ill-conditioned than $B_{zz}$, round-off error will be less severe in Alg.~\ref{alg:mixed_dir_cg} than in full CG because direct linear algebra is less vulnerable to bad conditioning. Also, Alg.~\ref{alg:mixed_dir_cg} is highly parallelizable because each column of $B_s$ can be computed completely independently of the others, while full CG is an inherently sequential algorithm.

\section{Numerical experiments} \label{num_expts}
In this section we show how semi-reduced methods can help us solve practical scientific problems faster and more robustly. To this end, we consider two model inverse problems relevant to scientific applications. In these problems, the use of Poissonian likelihoods and/or bound constraints greatly increases solution accuracy, so the unconstrained least squares is not preferable and reduced update methods are not appropriate. They are also well-suited for the linear algebra methods derived in \S \ref{sec:block_lin_algs}. The first problem is an exponential sum fitting problem involving multiple measurement vectors, and the second is a semiblind deconvolution problem from solar astronomy. We also solve a third problem, which is a toy model of the second problem. Its purpose is to show when trial point adjustment can be useful, since (as discussed below) we did not find it particularly useful in the first two problems.

For each of the three problems, we selected an appropriate semi-reduced method and compared it to a standard full update method. In the first two problems, the full update method was the projected Newton-type method Alg.~\ref{alg:tmp_damp} with no block Gaussian elimination and no block trial point adjustment. This approach was compared with two alternatives: Alg.~\ref{alg:tmp_damp} with elimination off and adjustment on, and Alg.~\ref{alg:tmp_damp} with elimination on and adjustment off. (Elimination and adjustment act independently, so testing a fourth condition with both techniques switched on yields little additional information.)  Block Gaussian elimination was performed using one of the methods derived in \S \ref{sec:block_lin_algs}, while block trial point adjustments were obtained by performing a few iterations of Alg.~\ref{alg:tmp_damp} to approximately solve $\min_{z \in \Zset} F(y^k,z)$, starting from the current iterate $z^k$. The parameters were set to $\delta = 10^{-4}$, $\alpha = 0.2$, $\lambda_{min} = 10^{-20}$, $\lambda_{max} = 10^{20}$, $\epsilon_0 = 2.2 \cdot 10^{-14}$, $\rho_{good} = 0.7$, $\rho_{bad} = 0.01$, $\tau = \max( 2.2 \cdot 10^{-15}, \norm{\Proj(x^0 - \nabla F(x^0)) - x^0} / 10^8)$ where $x^0$ is the initial point. In the third problem we use simpler algorithms which we describe later. All of our experiments were performed in MATLAB R2011a on a MacBook Pro with 2.4 GHz Intel Core 2 Duo processor. 

Our first finding was that trial point adjustment did not help us to solve the first two problems faster. Adjustment sometimes reduced backtracking and the total number of outer iterations needed, but not consistently or dramatically enough to outweigh the cost of solving inner adjustment subproblems at every iteration. As a result, total function evaluations and total runtime generally increased significantly when adjustment was used. For example, over 20 randomly sampled instances of the problem in \S \ref{exp_sum_fit}, we compared a stringent inner solver ($k_{max} = 100$, $\tau = \norm{\Proj(z^k - \nabla_z F(y^k,z^k)) - z^k} / 10^8$) to no inner solver at all; in the former case the total function evaluations to solve the problem ranged from $200-600$, while for no inner iterations the range was $60-100$. We tried various intermediates between these two extremes\textemdash intermediate values of $\tau$, lower values of $k_{max}$, stopping early if the Armijo condition was satisfied before the inner iteration limit\textemdash but we always found that it was most efficient to simply set $k_{max} = 0$, meaning no trial point adjustment.

For this reason we do not report any further on the effects of trial point adjustment in the first two problems. Instead we focus on the effects of block Gaussian elimination in the first two problems, and consider adjustment's effects only in the third problem. Note that since high-precision inner optimizations generally cause inefficiency in the first two problems, extensions of variable elimination that require them (such as \cite{SimaVanHuffel:varProBndMRS:2007, Nagy:nonnegVarPro:2012}) would be vulnerable to inefficiency in these problems, even if they could handle nonquadratic objectives. 

\subsection{Exponential sum fitting} \label{exp_sum_fit}
In exponential sum fitting problems, the expected value $\mu(t)$ of a physical quantity at time $t$ is assumed to be the sum of $c$ exponentially decaying components with decay rates $y_j$ and nonnegative weights $z_j$:
\beq
	\mu(t) = \sum_{j = 1}^c z_j \exp(-y_j t).
\eeq
In many cases the decay rates do not vary from experiment to experiment, but the weights $z$ may vary \cite{mullen2010sum}. Thus, if $n$ experiments are performed, the expected decay in the $k^{th}$ experiment is
\beq
	 \mu_k(t) = \sum_{j = 1}^c Z_{jk} \exp(-y_j t), \quad k = 1,\ldots,n.
\eeq
We assume that a set of $m$ Poisson-distributed observations $B_{1k},\ldots,B_{mk}$ of each $\mu_k(t)$ are made at $t = t_1,\ldots,t_m$:
\beq
	B_{ik} \sim \text{Poisson}(\mu_k(t_i)), \quad \text{for} \quad
	\begin{matrix} i = 1,\ldots,m \\ k = 1,\ldots,n. \end{matrix}
\eeq
If the columns of $B$ and $Z$ are stacked on top of each other to form vectors $b$ and $z$, then the associated maximum likelihood problem is
\beq
	\minimize_{y,z} L( (I_n \kron A(y))z) \quad \subto z \geq 0,
	\label{soe_prob_poiss}
\eeq
where $A(y)_{ij} = \exp(-y_j t_i)$, $\kron$ is the Kronecker product, $I_n$ is the $n \times n$ identity matrix, and $L(\mu)$ is the Poisson negative log-likelihood.

Using this model we generated synthetic data which simulated the problem of determining several decay rates from a large collection of relatively low-count time series. Each time series was generated from $c = 4$ decaying components with rates $(y_1,y_2,y_3,y_4) = (1,2,3,4)$ and $m = 1000$ uniformly spaced time samples from $t = 0$ to $5$. The number of measurement vectors was $n = 100$, and the nonnegative weights were randomly generated according to $z_{jk} = 10 \exp(1.2 \mathcal{Z}_{jk})$, where the $\mathcal{Z}_{jk}$ were random numbers from the standard normal distribution. A typical curve generated by this model is shown in Fig.~\ref{fig:decayExpt}. While this simple model does not directly represent a real physical problem, it generates problems similar in mathematical form, scale, and difficulty to problems encountered in real data analysis \cite{mullen2010sum,mullen2007timp}. In particular, each component has a few measurement vectors in which it dominates, but no component is ever observed in complete isolation. The persistent mixture of components with similar rates and the low signal-to-noise ratio combine to make this problem formidable.

As we mentioned above, trial point adjustment was not useful in this problem, so here we compare Alg.~\ref{alg:tmp_damp} in two modes: a semi-reduced mode with block Gaussian elimination, and a full update mode without it. In both cases the Hessian model was computed using the Gauss-Newton method \cite{Wedderburn:quasiLike:1974, Bunch:quasiLikeAlg:1993}, and the resulting Gauss-Newton system was solved by QR. (Direct Cholesky factorization of the normal equations is not sufficiently accurate due to the notoriously poor conditioning of exponential fitting problems \cite{Golub:separablenonlinear:2002}.) In the standard mode, the full normal equations were solved directly using MATLAB's built-in sparse QR routine,  while in the block Gaussian elimination mode, we used a MATLAB implementation of Alg.~\ref{qr_decomp_solve_block}. MATLAB sparse QR employs the state-of-the-art SuiteSparseQR package \cite{davis2011algorithm}. To obtain the best possible performance from SuiteSparseQR, matrix-vector products with the $Q$ factor were performed implicitly, and a permutation was applied to switch the blocks $J_y$ and $J_z$. (The permutation speeds up the algorithm by an order of magnitude, as it enables the underlying Householder triangularization method to preserve the matrix's sparsity pattern.) Note that Alg.~\ref{qr_decomp_solve_block} has a less efficient implementation than SuiteSparseQR because the loops in Alg.~\ref{qr_decomp_solve_block} run relatively inefficiently in MATLAB, while SuiteSparseQR is written in C++. 

Our main finding was that block Gaussian elimination computed steps several times faster than sparse QR with no loss of accuracy. In a typical random instance of the problem described above, step computation by sparse QR factorization of the full Jacobian required $0.38$ seconds (s), while Alg.~\ref{qr_decomp_solve_block} solved the system in $0.10$ s, a roughly 4-fold improvement. Since most of the algorithm's time is spent in step computation, the minimum was found significantly faster using block Gaussian elimination: in this instance, the standard mode took $18$ s, while using Alg.~\ref{qr_decomp_solve_block} took $6$ s. The accuracies of the two modes were functionally indistinguishable, as the objective values $F(y^k,z^k)$ output in each mode were the same to at least 8 significant figures. From this we infer that the two algorithms do essentially the same mathematical operations, but the computer finishes the operations faster using Alg.~\ref{qr_decomp_solve_block}.

The speed difference can be explained by two factors. First, Alg.~\ref{qr_decomp_solve_block} does not build the full $J$ matrix, but factorizes of the $n$ diagonal blocks of $J_z$ just as they are needed. In contrast, the sparse QR algorithm must build all of $J$ first, which takes $60-80\%$ of the CPU time required to actually solve the system. In Alg.~\ref{qr_decomp_solve_block} the blocks of $J_z$ are built and factorized just-in-time, so there is no need to build a large sparse matrix. Second, Alg.~\ref{qr_decomp_solve_block} solves the overall system by solving a large number of small and very similar subsystems, which is more CPU and cache-friendly than operating on a large sparse matrix. 

The formidable difficulty of this problem, and the need for a bound-constrained Poissonian solver, may be appreciated by comparing the accuracy of the Poissonian method to a popular alternative for Poissonian problems, the variance-weighted least squares method. In the variance-weighted least squares method one solves
\beq
	\minimize_{y,z} \norm{W[(I_n \kron A(y))z - b]}_2^2 \quad \subto z \geq 0,
	\label{soe_prob_nnls}
\eeq
where $W$ is a diagonal matrix with $W_{ii} = 1/\max(b_i^{1/2},\epsilon)$, and $\epsilon = 1$ is a small constant used to avoid division by zero \cite{Laurence:2010:mlePoisson}. We generated 100 random instances of the exponential sum fitting problem described above, and solved each using the Poissonian approach \eqref{soe_prob_poiss} and the weighted least squares approach \eqref{soe_prob_nnls}, in both cases using Alg.~\ref{alg:tmp_damp}. The decay vector $y$ resulting from each experiment was sorted to account for the problem's permutation ambiguity, resulting in 100 estimates of $y_1, y_2, y_3,$ and $y_4$ from each method. We then calculated the median and median absolute deviation of the 100 estimates of each $y_i$ from each method. (We used the median as a summary statistic because it is invariant to reparametrization of $y$ and robust to the occasional failures of both methods.) The results are shown in Fig.~\ref{fig:decayExpt}, \emph{right}, and it is clear that the Poissonian solver's decay rates are far more accurate.

\begin{figure}
\centering
\begin{tabular}{c c}
\includegraphics[scale=0.72]{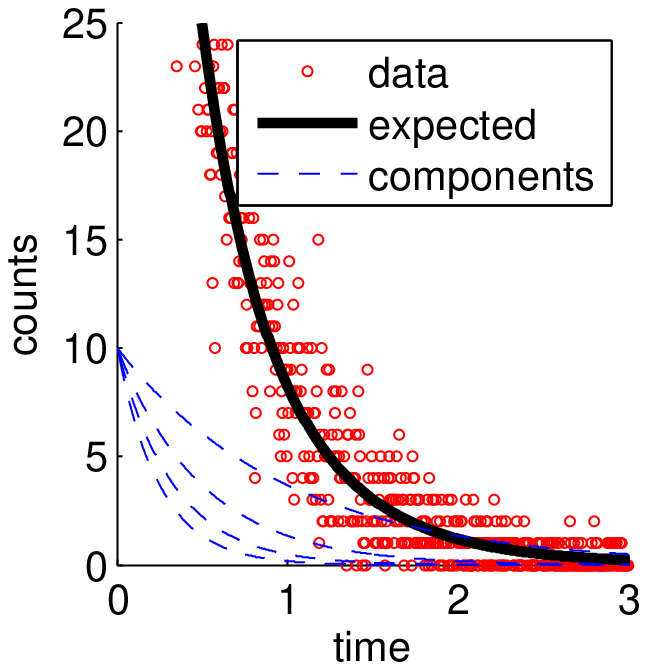} & \includegraphics[scale=0.60]{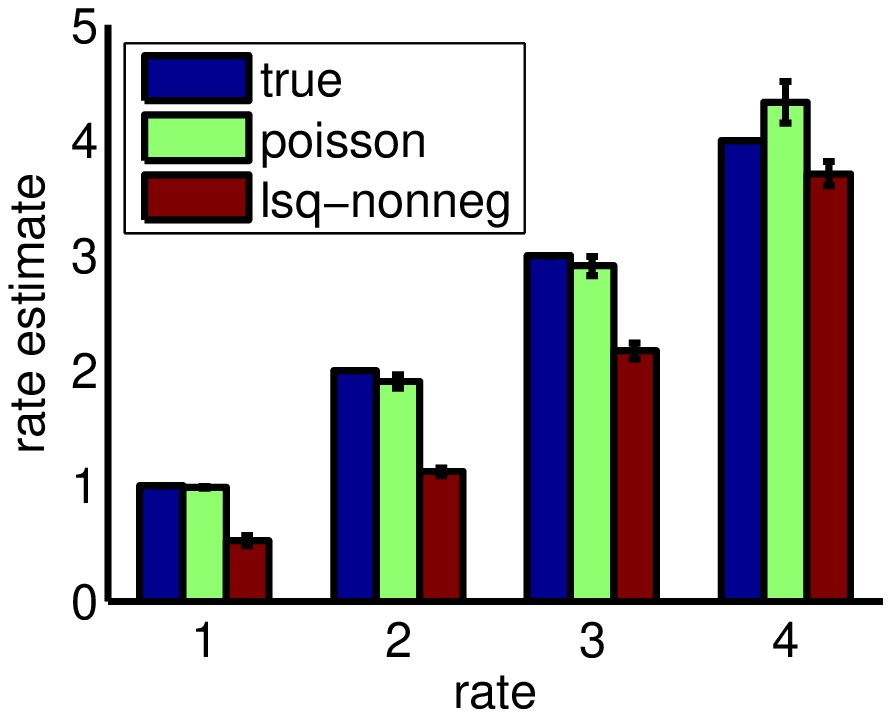} 
\end{tabular}
\caption{\emph{Left:} Sample data from the sum-of-exponentials model. The four decaying components (blue dotted lines) have decay rates $y_j = j$ for $j = 1,2,3,4$, and when summed together with weights $z_j$, these components create the expected intensity curve $\mu(t)$ (solid black line). The poisson-distributed samples $b_i$ of $\mu(t)$ (red dots) are taken at a spacing of $\Delta t = 0.005$. The low available counts suggest a Poisson likelihood should be used. \emph{Right:} Comparison of fitted and true decay rates $y_j$ for $j = 1,2,3,4$ using variance-weighted nonnegative least squares and Poisson likelihood. The bar heights are the median values found by solving 100 random problem instances, and the error bars represent median absolute deviations. }
\label{fig:decayExpt}
\end{figure}

\subsection{Multiframe semiblind deconvolution} \label{semiblind}
Image deconvolution is a linear inverse problem in which we have an image $b$ degraded by convolution with a known point spread function (PSF) $h$, and we wish to undo the degradation to obtain the unknown clean image $z$. Assuming that each of these variables are 2D arrays supported on a square $\Omega \in \mathbb{Z}^2$, we can write the problem as
\beq
	Az + \epsilon = b,
\eeq
where $A : \reals^\Omega \rightarrow \reals^\Omega$ is the convolution operator: $Az = h * z$, and we assume periodic boundary conditions for simplicity. In multiframe blind deconvolution, there are several images and PSFs and the PSFs depend on unknown parameters, so that we have
 \beq
 	A(y^{(k)}) z^{(k)} + \epsilon^{(k)} = b^{(k)} \quad \text{for $k = 1,\ldots,n$}.
 \eeq
If we have a parametric model of the PSFs, the problem is called semiblind.

Here we consider a simplified, synthetic version of a real multiframe semiblind deconvolution problem from solar imaging, which is described in \cite{shearer2012first}. In this problem, a spaceborne telescope observing the Sun in the extreme ultraviolet wavelengths collects images which are are contaminated by stray light. The stray light effect is well-modeled by convolution with a single unknown parametric PSF. The telescope observes $n$ images of the Moon transiting in front of the Sun, and while the Moon does not emit in the extreme ultraviolet (Fig.~\ref{fig:solar_prob}, \emph{top middle}), stray light from the Sun spills into the lunar disk (\emph{bottom middle}). Given the supports $M^{(k)} \subset \Omega $ of the lunar disks within each image, our task is to determine the PSF by solving
 \begin{equation}
 	\minimize_{y,\{z^{(k)}\}} \sum_{k = 1}^n \norm{A(y) z^{(k)} - b^{(k)}}^2 \quad \text{subject to} \,\,
	\begin{matrix} 
	&z^{(k)} \geq 0 \\ 
	&z^{(k)}_{M^{(k)}} = 0
	\end{matrix} 
	\quad \text{for $k = 1,\ldots,n$}.
	\label{solar_prob}
\end{equation}
The PSF is modeled using two components. The PSF core is modeled by a single pixel with unknown value $\alpha \in (1/2,1]$, while the wings are modeled by a radially symmetric piecewise power law $p_{\bSeries}(r)$ depending on unknown parameters $\bSeries$:
\begin{equation} \label{hmrBuild}
	h_y(v) = \alpha \delta_0(v) + (1-\alpha) p_{\bSeries}(\norm{v}_2), \quad \text{for} \,  v \in \Omega,
\end{equation}
where $\delta_0$ is the Kronecker delta. To define the piecewise power law, we set $p_{\bSeries}(0) = 0$, then for $r > 0$ we set logarithmically spaced breakpoints $(r_i)_{i=0}^S$ defining $S = 12$ subintervals, starting from $r_0 = 1$ and ending at $r_S = \frac{\sqrt{2}}{2} s$ where $s$ is the sidelength of the square $\Omega$. On each subinterval $[r_{i-1},r_i)$, the formula is given by $p_{\bSeries}(r) \propto r^{-\beta_i}$, where $\beta_i \geq 0$, and $\bSeries = (\beta_1,\ldots,\beta_S)$. The proportionality constants are determined by a continuity constraint between subintervals and the normalization constraint $\sum_v p_{\bSeries}(\norm{v}) = 1$. The free parameters of the PSF model $h_y$ are then $y = (\alpha,\bSeries) \in \reals^{N_y}$, where $N_y = 1+S = 13$. The true profile $p_{\bSeries}(r)$ was generated using $\bSeries$ values similar to those in \cite{shearer2012first}, and is shown in log-log scale in (Fig.~\ref{fig:solar_prob}, \emph{top left}), with the resulting PSF directly below.

\begin{figure}
\centering
\includegraphics[scale=1.0]{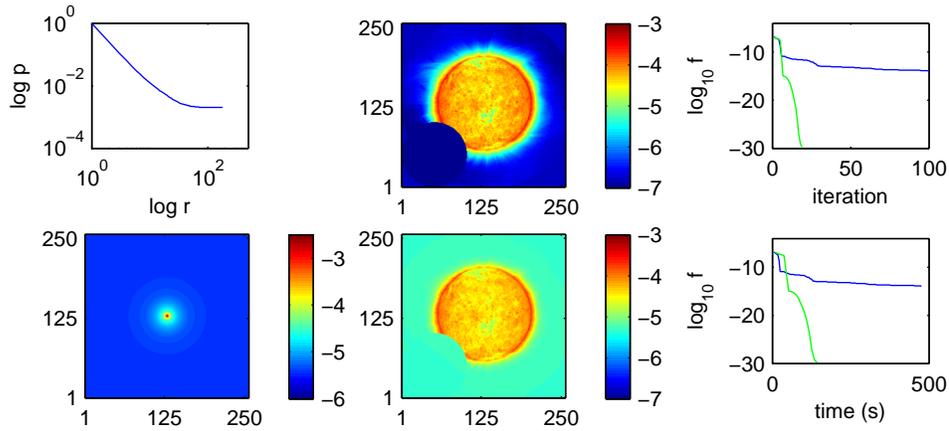}
\caption{Overview of the solar semiblind deconvolution experiment. \emph{Top left:} The ground truth PSF profile $p^{true}(r)$ in log-log scale, where it is piecewise linear. \emph{Bottom left:} The ground truth PSF generated by the profile above. \emph{Top middle:} one of the three clean lunar transit images, with lunar disk in the bottom left corner. \emph{Bottom middle:} the observed image formed by convolving the top image with the PSF. \emph{Right:} semilog plot of objective versus iteration \emph{(top)} and CPU time \emph{(bottom)} for the standard mode of Alg.~\ref{alg:tmp_damp} and the mode employing the mixed CG/Direct method.}
\label{fig:solar_prob}
\end{figure}

We used data from the STEREO-EUVI satellite to generate $n = 3$ synthetic lunar transit images of size $256 \times 256$. To simulate the Moon's transit, a disk of pixels was set to zero in each image. These images were convolved with the ground truth PSF to create the blurry observations, and no noise was added for simplicity. (Noise is not a very important issue in this problem because the real data have very little noise at this resolution, deconvolution with the PSF is quite well-conditioned for $\alpha > 1/2$, and the expected range of $\alpha$ is well above this.)

As before, Alg.~\ref{alg:tmp_damp} was run in two modes: a semi-reduced mode employing Gaussian elimination, and a standard one without. In the standard mode of Alg.~\ref{alg:tmp_damp}, the search direction was calculated by CG on the full system $B \Delta x = -g$.  Preliminary experiments revealed that the full CG algorithm was very slow. However the situation improved substantially when we used a scalar preconditioner $c I$ on the $y$ block, where $c = 10^5$ was found to work well. All CG iterations were stopped at a relative residual tolerance of $10^{-6}$ and maximum iteration ceiling $40$, as these values worked relatively well for the full CG algorithm. In the block Gaussian elimination mode, the search direction was calculated using the mixed CG/Direct algorithm, Alg.~\ref{alg:mixed_dir_cg}. The mixed CG/Direct algorithm required no special tuning or preconditioners.

Our main finding was that the block Gaussian elimination mode using Alg.~\ref{alg:mixed_dir_cg} converged quite quickly and robustly, while the standard mode experienced a long period of sluggish convergence after an initially fast descent (Fig.~\ref{fig:solar_prob}, \emph{right}). The average CPU time per step was about the same in each of the two modes, so we can attribute the block Gaussian elimination mode's superior performance to better search directions, which enabled convergence in far fewer iterations than the full CG mode.

The better search directions of the block Gaussian elimination mode can be explained by considering the unusual spectrum of the Gauss-Newton Hessian. It has two very different components: a large cluster of $\approx N_z$ near-unity eigenvalues due to the very well-conditioned $B_{zz}$ block, and a sprinkling of $\approx N_y$ eigenvalues contributed by the other three blocks. The latter are, to put it lightly, less tame: they can easily spread over 15 orders of magnitude and move unpredictably as the iterations progress.

Naively, we might expect full CG to make short work of such a system. We simply apply a scalar preconditioner to the badly behaved blocks involving $\Delta y$, pushing the $N_y$ scattered eigenvalues to lie above the $N_z$ cluster. Then the spectral theory of CG predicts convergence in $N_y + k_z$ iterations, where $k_z$ is the number of iterations required to make the CG spectral polynomial nearly zero on the $N_z$ cluster \cite{trefethen1997numerical}. We expect $k_z$ to be small because the $N_z$ cluster is very tightly centered around unity.

In practice, however, it is difficult to know in advance where the mobile eigenvalues will be, and their enormous spread raises issues of rounding error. Thus it is difficult to get good solutions out of full CG, and the search directions suffer, causing sluggish convergence. In constrast, the mixed CG/Direct algorithm applies CG to the well-conditioned $B_{zz}$ block alone, and deals with the other blocks by direct linear algebra. Since direct linear algebra is much less susceptible than CG to ill-conditioning and rounding error, the result is high-quality search directions and quick convergence. 

\subsection{A model semiblind deconvolution problem for block trial point adjustment} \label{model_adjustment_prob}
Given the failure of trial point adjustment to speed up the solution of the previous two problems, the reader may wonder if it has any application beyond its theoretical role in the connection between full and reduced update methods. The literature suggests that adjustment certainly can increase convergence rate and robustness \cite{lanzkron1996analysis,smyth1996partitioned,o2012variable,ruhe1980algorithms,gilles2002computational}. However the speed gains relative to standard methods are highly variable: adjusted methods are slower in our experiments, a factor of 2 or 3 times faster in certain image processing problems, and multiple orders of magnitude faster in some difficult curve fitting problems. Clearly adjustments must be adapted to the problem at hand, but it is difficult to predict when it will be useful. Here we present a toy semiblind deconvolution problem similar to the one solved in the previous section, and show that trial point adjustment is valuable for solving this problem in the most difficult cases. 

As in the solar problem, our toy problem involves semiblind deconvolution of an extended, uniformly bright object which has been convolved with a long-range kernel. The true image $u^t$ and kernel $h^t$ are both 1-D signals of length $m$ supported on $\{-j,\ldots,j\}$, where $m = 2j+1$. They come from single-parameter signal families given by $h_y(p) = y \delta_0(i) + (1-y) \frac{1}{m} \mathbf{1}(i)$ and $u_z(i) = z \cdot \mathbf{1}_S(i)$, where $\mathbf{1}_S(i)$ is the indicator for the set $S = \{-\ell,\ldots,\ell\}$ of size $s = 2 \ell + 1$, $\mathbf{1}(i)$ is the constant ones function. Letting $(y^t,z^t)$ denote the unknown true parameter values, the problem is to determine $(y^t,z^t)$ from the blurry observation $f = h^t * u^t$, where periodic convolution and no noise is assumed. 
The values of $y^t$ and $z^t$ can be found by minimizing the difference between $h_y * u_z$ and $f$ with respect to some loss function, which we choose as the Huber loss
\beq
	\ell(x) = 
	\begin{cases}
	\tfrac{1}{2} x^2 & |x| \leq t \\
	t(|x| - \tfrac{s}{2}) & |x| > t
	\end{cases}
\eeq
with threshold $t = 0.3$. We choose this loss function simply because it is a common nonquadratic loss and the optimization phenomenon of interest occurs when it is used.  Noting that physically we must have $0 \leq y \leq 1$ and $z \geq 0$, we obtain the optimization problem
\beq
	\min_{0 \leq y \leq 1, z \geq 0} \bigg\{ F(y,z) \triangleq \sum_i \ell((h_y * u_z - f)_i)  \bigg\}.
	\label{fullObjToy}
\eeq
A simple formula for $F(y,z)$ can be found by observing that both the prediction $h_y * u_z$ and $f$ take only two values. Letting $\rho = s/m$ be the ratio of the object support to signal size, $q_1(y,z) = y z + (1 - y z) \rho$ the predicted value of the blurry image on the support, $q_2(y,z) = (1-y)z \rho$ the predicted value off the support, and $q_i^t = q_i(y^t,z^t)$ for $i = 1,2$, we have
\begin{equation}
	F(y,z) = \frac{m}{2} \left( \rho \ell(q_1(y,z) - q_1^t) + (1-\rho) \ell(q_2(y,z) - q_2^t) \right).
\end{equation}
The $m/2$ scale factor does not affect the location of the minimum nor the path of any of the optimization algorithms we consider here. Therefore, for our purposes, the parameter $\rho$ is effectively the only free parameter in the problem family. We use the values $\rho = 10^{-2}, 10^{-6}$ to create two objectives whose graphs are depicted in Fig.~\ref{fig:jan_comp}, \emph{far left}. As $\rho \rightarrow 0$, the term $\rho \ell(q_1(y,z) - q_1^t)$ vanishes, the $(1-\rho)\ell(q_2(y,z) - q_2^t)$ becomes dominant, and the objective landscape becomes a narrow, hyperbolic trench.

We solved this problem at both values of $\rho$, and for each value we used Alg.~4 in a full update mode (without trial point adjustment) and a semi-reduced mode (with trial point adjustment). In the latter case, the block trial point adjustment used was a single iteration of Alg.~4 to minimize $F(y,z)$ in $z$ with $y$ fixed. Algorithm parameters were chosen as in the previous section. 

The paths taken by the full and semi-reduced methods are shown in Fig.~\ref{fig:jan_comp}, \emph{center left and right}. We observe that the methods take nearly identical paths when $\rho = 10^{-2}$, but when $\rho = 10^{-6}$ the full update method is forced to take very small steps. At \emph{far right}, the distance to the minimum, $\norm{(y^k,z^k) - (y^t,z^t)}_2$, is plotted versus iteration $k$ for each method. The superior convergence rate of the semi-reduced method is clear when $\rho = 10^{-6}$. 

The behavior of each algorithm can be understood by considering the geometry of the steps it takes. The full update method takes steps along straight lines. Straight lines cannot follow a curved trench for long, so there is an upper bound on the size of an admissible step. As $\rho \rightarrow 0$, the trench tightens and the admissible steps become very small, so that progress is very slow. The semi-reduced method takes a `dogleg' step as illustrated in Fig.~\ref{visual_mod}, which enables it to stay in the valley. 

To avoid the small admissible step issue that stymies the full method, it is critical that adjustment be done \emph{before the trial point is evaluated}. This is the key feature distinguishing semi-reduced methods from other methods, such as simple alternation between a full update and a partial update. Other strategies, such as nonmonotone line search \cite{zhang2004nonmonotone} and greedy two-step methods \cite{conn:twoStep:1999}, have a similar step structure and could also work on this problem; however it is unclear if they can match the semi-reduced method's complete insensitivity to the value of $\rho$. 

It is important to note that the phenomenon we have described here does not occur for all loss functions $\ell(x)$. For example, we found that if the Poisson log-likelihood is used, the objective landscape does not have such a tight curved valley, the full update method solves the problem quite efficiently, and the semi-reduced method's inner iterations expend effort without benefit. Curved valleys are thus an occasional problem with potentially severe efficiency consequences. The semi-reduced framework seems appropriate for dealing with such a problem, since one has the option to perform inner descent iterations only when necessary.

\begin{figure}
\centering
\includegraphics[scale=1.0]{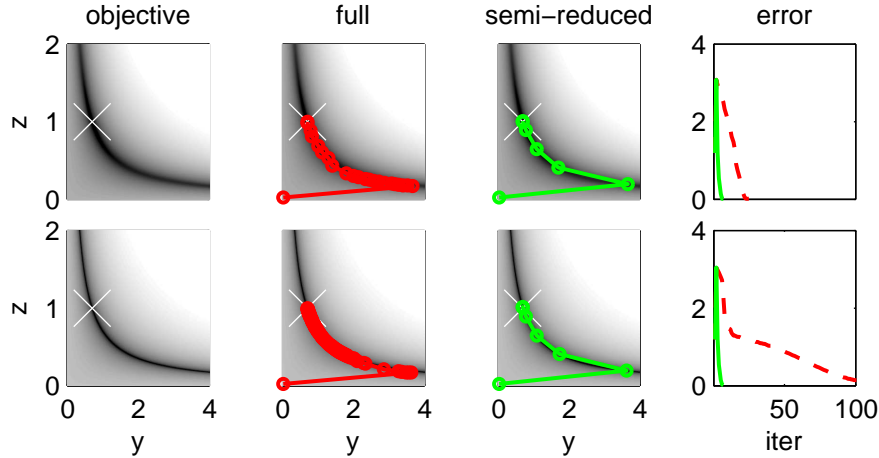}
\caption{Comparison of alternating, joint, and nested descent strategies on toy blind deconvolution problem. \emph{Row 1:} Plots of $F(y,z)$ for $\rho = 10^{-2}$ (top) and $10^{-6}$ (bottom), logarithmic greyscale. The white crosses mark the minimum at $(y^t,z^t) = (0.7,1)$, where $F = 0$. \emph{Center left and right:} The iterates of the full and semi-reduced methods for each $\rho$ value, starting from $(y^0,z^0) = (0.02,0.02)$. \emph{Far right:} semilogarithmic plot of the error $\norm{(y^k,z^k) - (y^t,z^t)}_2$ versus iteration $k$ for the full method (dashed line) and semi-reduced method (solid line). \vspace{-3mm}}
\label{fig:jan_comp}
\end{figure}

\section{Conclusion}

Reduced update optimization methods, which are based on variable elimination, have been found to be particularly fast and robust in certain difficult separable inverse problems. Unfortunately, using them in problems beyond unconstrained least squares presents serious theoretical and practical difficulties, in particular the need for expensive optimal trial point adjustments and complex derivatives of a reduced functional. We have described a new class of \emph{semi-reduced} methods which interpolate between full and reduced methods. Semi-reduced methods share the desirable characteristics of reduced methods while being flexible enough to avoid their downsides. A key advantage of the semi-reduced framework is the flexibility to use adjustments where they are useful and avoid them where they are not, all within a single convergent method.

We began by reinterpreting reduced methods as full update methods that have been modified and simplified. We showed that if \emph{block Gaussian elimination} and an optimal \emph{block trial point adjustment} are used within a full update method, the adjustment's optimality renders certain computations unnecessary. Removing these unnecessary computations yields a simplified method that turns out to be equivalent to a reduced method. To confirm that this reinterpretation of reduced update methods is correct and generally applicable, we derived the well-known reduced update Newton and variable projection methods using our modification and simplification process. We defined semi-reduced methods by omitting the final simplifications, which frees us from the need to perform expensive optimal block trial point adjustments. We then incorporated block Gaussian elimination and trial point adjustment into an algorithm for general bound constrained problems, and showed that its convergence follows almost immediately from the convergence theorem for the original method. Finally, we showed that many of the structure-exploiting properties of variable elimination can be obtained by using appropriate block Gaussian elimination algorithms.

Block Gaussian elimination is suited for problems where the Hessian model's $B_{zz}$ subblock is block diagonal, circulant, banded, or has other exploitable structure. We described two situations where we expected block Gaussian elimination to outperform a standard all-at-once method, and these expectations were borne out in numerical experiments on realistic problems derived from the scientific inverse problem literature. It is notable that both of the methods we presented involve the solution of many independent subproblems and are thus ideal candidates for parallelization.

Block trial point adjustment is appropriate when we expect the graph of $F(y,z)$ to contain a narrow, curved valley. Trial points from full update methods tend to leave the valley and thus will be rejected unless a trial point adjustment is used to return to the valley. In our first two numerical experiments trial point adjustments turned out to be computationally wasteful, so it was critical that we had the flexibility to perform suboptimal adjustments or even none at all (which turned out to be the best option). In our third experiment we presented a reasonable toy inverse problem where the curved valley effect was significant enough to warrant trial point adjustment, but the parameter values where this occurred were somewhat extreme. Since the curved valley effect is important in some real problems \cite{lanzkron1996analysis,conn:twoStep:1999}, a better understanding of precisely when it occurs would be useful. 

\section{References}

\bibliographystyle{unsrt}
\bibliography{master_bib}

\end{document}